\documentclass[a4paper,12pt,reqno,oneside]{amsart}
\usepackage{setspace} 
\usepackage{typearea} 

\usepackage{amscd,amsmath,amssymb,verbatim}
\usepackage{xr-hyper} 
\usepackage{graphicx}
\usepackage{ifthen}
\usepackage{paralist,cite}
\usepackage[colorlinks,backref,final,bookmarksnumbered,bookmarks]{hyperref}
\usepackage[backrefs]{amsrefs}
\usepackage{tikz-cd}
\usepackage{enumitem}

\usepackage[utf8]{inputenc}
\usepackage[T1,T2A]{fontenc}

\newcommand{\arxiv}[2][]{\ifthenelse{\equal{#1}{}}
{\href{http://arxiv.org/abs/#2}{\tt arXiv:#2}}
{\href{http://arxiv.org/abs/math/#2}{\tt arXiv:math.#1/#2}}}

\theoremstyle{plain}
\newtheorem{maintheorem}{Theorem}
 
\newtheorem{theorem}{Theorem}[section]
\newtheorem{lemma}[theorem]{Lemma}
\newtheorem{corollary}[theorem]{Corollary}
\newtheorem{proposition}[theorem]{Proposition}
\newtheorem{problem}[theorem]{Problem}
\newtheorem*{corollary*}{Corollary}

\theoremstyle{definition}
\newtheorem{example}[theorem]{Example}

\newtheoremstyle{remark}
{}{}{}{}{\itshape}{}{ }{\thmname{#1}\thmnumber{ \itshape #2.}}
\theoremstyle{remark}
\newtheorem{remark}[theorem]{Remark}

\newtheoremstyle{concise}
{}{}{}{}{\bfseries}{}{ }{\thmnumber{#2.}\thmnote{ #3.}}
\theoremstyle{concise}

\def\N{\mathbb{N}} 
\def\R{\mathbb{R}} 
\def\Z{\mathbb{Z}}

\def\x{\times}
\def\but{\setminus} 
\def\emb{\hookrightarrow}

\def\eps{\varepsilon} 
\def\phi{\varphi} 

\def\xr#1{\xrightarrow{#1}} 
 \renewcommand{\:}{\colon}

\DeclareMathOperator*{\colim}{colim} 
\DeclareMathOperator*{\derlim}{lim^1}

\DeclareMathOperator{\Int}{Int} \DeclareMathOperator{\id}{id}
 \DeclareMathOperator{\Cl}{Cl} 
\DeclareMathOperator{\Fr}{Fr}

\def\STAR{\raisebox{-1pt}{\LARGE$\star$}}
\def\Star{{\raisebox{1pt}{\hspace{0.5pt}\normalsize$\star$}}}
\def\Bullet{{\raisebox{2pt}{\hspace{-1pt}\Huge$.$}}}
\def\join{\mathop{\raisebox{-3pt}{\Huge$*$}}}

\def\tph#1{\raise2.5pt\hbox{\the\textfont1\char"7F}\!\!#1}
\def\tpm#1{\raise0pt\hbox{\the\textfont1\char"7F}\!#1}
\def\tpl#1{\lower1.5pt\hbox{\the\textfont1\char"7F}\!#1}

\def\bydef{\mathrel{\mathop:}=}

\DeclareSymbolFont{bskadd}{U}{bskma}{m}{n}
\DeclareFontFamily{U}{bskma}{\skewchar\font130 }
\DeclareFontShape{U}{bskma}{m}{n}{<->bskma10}{}
\DeclareMathSymbol{\varlrttriangle}{\mathord}{bskadd}    {"E4}

\makeatletter
\newcommand*\nullseq{\mathop{\mathpalette\@biguoperator{\bigsqcup}}}
\newcommand*\@biguoperator[2]{\ooalign{\hidewidth$#1$\raisebox{1pt}{\scriptsize$\star$}\hidewidth\cr$#1#2$\cr}}
\makeatother

\begin{document}

\title{Coronated polyhedra and coronated ANRs}
\author{Sergey A. Melikhov}
\address{Steklov Mathematical Institute of Russian Academy of Sciences,
ul.\ Gubkina 8, Moscow, 119991 Russia}
\email{melikhov@mi-ras.ru}

\begin{abstract}
Locally compact separable metrizable spaces are characterized among all metrizable spaces as those that admit
a cofinal sequence $K_1\subset K_2\subset\cdots$ of compact subsets.
Their \v Cech cohomology is well-understood due to Petkova's short exact sequence 
$0\to\derlim H^{n-1}(K_i)\to H^n(X)\to\lim H^n(K_i)\to 0$.

We study a dual class of spaces.
We call a metrizable space $X$ a {\it coronated polyhedron} if it contains a compactum $K$ such that $X\but K$ is 
a polyhedron.
These include, apart from compacta and polyhedra, spaces such as the topologist's sine curve 
(or the Warsaw circle) and the comb (=comb-and-flea) space.
The complement of every locally compact subset of $S^n$ is a coronated polyhedron.

We prove that a metrizable space $X$ is a coronated polyhedron if and only if it admits a countable polyhedral
resolution; or, equivalently, a sequential polyhedral resolution $\dots\to R_2\to R_1$.
In the latter case, we establish a short exact sequence $0\to\derlim H_{n+1}(R_i)\to H_n(X)\to\lim H_n(R_i)\to 0$
for Steenrod--Sitnikov homology and also for any (extraordinary) homology theory satisfying 
Milnor's axioms of map excision and $\prod$-additivity.
We also show that such homology theories are invariants of strong shape for coronated polyhedra.
On the other hand, Quigley's short exact sequence $0\to\derlim\pi_{n+1}(R_i)\to\pi_n(X)\to\lim\pi_n(R_i)\to 0$ 
for Steenrod homotopy of compacta fails for Steenrod--Sitnikov homotopy of coronated polyhedra,
at least when $n=0$.
\end{abstract}

\maketitle
\section{Introduction}

It is well-known that the poset $\kappa_X$ of compact subsets of a metrizable space $X$ 
(ordered by inclusion) contains a cofinal sequence if and only if $X$ is locally compact and separable
(see \cite{M00}*{Proposition \ref{book:local compactum}}).
Moreover, a sequence of compact subsets $K_1\subset K_2\subset\dots$ of $X$ is cofinal in $\kappa_X$
if and only if each $K_i$ lies in the interior of some $K_j$ 
(see \cite{M00}*{Proof of Proposition \ref{book:local compactum}}).

\begin{theorem}\label{petkova} {\rm (Petkova \cite{Pe1}*{Proposition 4})}
Let $X$ be a locally compact separable metrizable space.
If $X$ is the union of its compact subsets $K_1\subset K_2\subset\dots$, where each $K_i\subset\Int K_{i+1}$,
then its \v Cech cohomology fits in a natural short exact sequence
\[0\to\derlim H^{n-1}(K_i)\to H^n(X)\to\lim H^n(K_i)\to 0.\]
\end{theorem}

This result was also rediscovered by Massey \cite{Mas}*{Theorem 4.22}.
We refer to \cite{M00} for a proof and further references and discussion.

The purpose of the present paper is to obtain a dual theorem for homology.
Or rather to understand the natural class of metrizable spaces for which it should be valid.

A class of metrizable spaces that is straightforwardly dual to locally compact separable metrizable spaces
is well-known.
Namely, metrizable spaces whose directed set of open covers (ordered by refinement) contains 
a cofinal sequence are characterized as follows.

\begin{theorem} \label{uc} {\rm (Nagata \cite{Nag}, Monteiro--Peixoto \cite{MoP}, Levshenko \cite{Lev}, 
Atsuji \cite{At}, Rainwater \cite{Rai}; see also \cite{Mrow}, \cite{Bee})}

(a) The following are equivalent for a metrizable space $X$:
\begin{itemize}
\item the directed set of open covers of $X$ contains a cofinal sequence;
\item the set $X'$ of cluster points of $X$ is compact;
\item the finest uniformity on $X$ is metrizable;
\item every Hausdorff quotient of $X$ is metrizable;
\item $X$ admits a metric satisfying any of the equivalent conditions in (b) below.
\end{itemize}

(b) The following are equivalent for a metric space $M$:
\begin{itemize}
\item every continuous map of $M$ into any metric space is uniformly continuous;
\item $d(A,B)>0$ for any disjoint closed $A,B\subset M$;
\item every open cover of $M$ has a positive Lebesgue number;
\item the set $M'$ of cluster points of $M$ is compact, and for each $\eps>0$, the complement of
the $\eps$-neighborhood of $M'$ is uniformly discrete.
\end{itemize}
\end{theorem}

Let us note that a metrizable space $X$ is compact if and only if {\it every} metric on $X$ satisfies 
the conditions in (b) \cite{Mrow}.
Metric spaces satisfying the conditions of (b) are called {\it UC spaces} (and also ``Atsuji spaces'' and
``Lebesgue spaces''), and so metrizable spaces satisfying the conditions of (a) are called 
{\it UC-metrizable}.
Thus a metrizable space $X$ is UC-metrizable if and only if it contains a compactum $K$ such that $X\but K$
is discrete.

\begin{example} \label{uc nullseq}
Let $\N=\{1,2,\dots\}$ with the discrete topology.
The null-sequence%
\footnote{The {\it null-sequence} $\nullseq_{i=1}^\infty X_i$ is the subset 
$\bigcup_{i=1}^\infty pt^{i-1}\x X_i\x pt^\infty$ of the product $\prod_{i=1}^\infty (X_i\sqcup pt)$.
Alternatively, $\nullseq_{i=1}^\infty X_i$ can be defined as the limit of the inverse sequence
$\dots\to X_1\sqcup X_2\sqcup X_3\sqcup pt\to X_1\sqcup X_2\sqcup pt\to X_1\sqcup pt$
(see \cite{M00}*{\S\ref{book:metric wedge}}).}
$\nullseq_{i=1}^\infty\N$ of copies of $\N$ is UC-metrizable (since it has only one cluster point)
but is not locally compact.
Clearly, $\nullseq_{i=1}^\infty\N$ is homeomorphic to the subset 
$\{(\frac1m,\frac1{mn})\mid m,n\in\N\}\cup\{(0,0)\}$ of the plane.
\end{example} 

Every discrete set may be understood as a $0$-dimensional polyhedron.
We refer to \cite{M00} for a treatment of polyhedra.
In particular, our polyhedra are always endowed with the metric topology (and not the weak topology).

\subsection{Coronated polyhedra}
We call a metrizable space $X$ a {\it coronated polyhedron}%
\footnote{Initially the author spoke of ``compactohedra'' (see version 1 of the arXiv preprint \cite{M-III}). 
But later it became clear that to better understand coronated polyhedra it helps to study coronated ANRs as well 
(as Theorem \ref{main-poly} depends on Theorem \ref{main-ANR}), which was one of the reasons to change the terminology.}
if it contains a compactum $K$ such that $X\but K$ is a polyhedron.
Let us note that if $V$ is the set of vertices of some triangulation of the polyhedron $X\but K$, then 
$K\cup V$ is UC-metrizable.

Clearly, coronated polyhedra include compacta and polyhedra.

\begin{example}
If $X$ is a locally compact subset of $S^n$, then $S^n\but X$ is a coronated polyhedron 
(see Theorem \ref{co-locally-compact}(a)).
\end{example}

\begin{example} Some standard examples in introductory courses of topology, such as the ``comb and flea'' 
space $\{\frac1n\mid n\in\N\}\x[0,1]\cup(0,1]\x\{1\}\cup\{(0,0)\}$ and the ``topologist's sine curve'', 
$\{(x,\sin\frac1x)\mid x\in (0,1]\}\cup\{0\}\x[-1,1]$, are coronated polyhedra.
\end{example}

\begin{example} \label{nullseq polycom}
The null-sequence $\nullseq_{i\in\N} P_i$ of arbitrary polyhedra $P_i$ is a coronated polyhedron.
If each $P_i$ is non-compact, then $\nullseq_{i\in\N} P_i$ is not locally compact.
\end{example} 

Some further examples of coronated polyhedra are discussed in \S\ref{examples}.

\begin{maintheorem}\label{main-poly} The following are equivalent for a metrizable space $X$:
\begin{enumerate} 
\item $X$ is a coronated polyhedron;
\item $X$ admits a countable resolution consisting of polyhedra and polyhedral maps;
\item $X$ admits a resolution of the form $\dots\xr{p_2} R_2\xr{p_1}R_1$, where each $R_i$ 
is a polyhedron and each $p_i$ is a polyhedral map.
\end{enumerate}
\end{maintheorem}

Resolutions are a basic concept of shape theory (cf.\ \cite{Mard}, \cite{MS}).
We refer to \cite{M00} for a concise treatment of resolutions in the metrizable case.

Resolutions of non-compact spaces are, generally speaking, inverse systems indexed by uncountable 
directed sets.
The author is unaware of any previous literature discussing sequential (or countable) ANR resolutions for
any spaces other than ANRs and UC-metrizable spaces.

Let us note that (3) trivially implies (2) in Theorem \ref{main-poly}.
The converse implication follows since every countable directed set contains a cofinal sequence, and 
passage to a cofinal subset preserves the property of being a resolution 
(see \cite{M00}*{Lemmas \ref{book:chain} and \ref{book:cofinal-resolution}}).
The equivalence of (1) and (2) is proved in Corollary \ref{poly-main}, which also contains some additional
characterizations of coronated polyhedra in terms of inverse sequences.

\subsection{Coronated ANRs}
We call a metrizable space $X$ a {\it coronated ANR} if it contains a compactum $K$ such that 
$X\but K$ is an ANR.

\begin{example}
If $X$ is a locally compact subset of the Hilbert cube $I^\infty$, then $I^\infty\but X$ is a coronated ANR 
(see Theorem \ref{co-locally-compact}(b)).
\end{example}

\begin{example} \label{nullseq anr}
The null-sequence $\nullseq_{i\in\N} P_i$ of arbitrary ANRs $P_i$ is a coronated ANR.
\end{example}

Some further examples of coronated ANRs are discussed in \S\ref{examples}.

\begin{maintheorem}\label{main-ANR} The following are equivalent for a metrizable space $X$:
\begin{enumerate}
\item $X$ is a coronated ANR;
\item $X$ admits a resolution $\dots\to X_2\to X_1$, where the image of each $X_i$ in each $X_j$ is an ANR;
\item $X$ admits a sequential ANR resolution whose bonding maps are embeddings.
\end{enumerate}
\end{maintheorem}

\begin{problem}
Suppose that a metrizable space $X$ admits a resolution $\dots\to X_2\to X_1$ where each $X_i$ is an ANR.
Is $X$ a coronated ANR?
\end{problem}

\subsection{Steenrod--Sitnikov homology and homotopy}

Using Theorem \ref{main-poly}, or rather its more refined version, we obtain the following dualization 
of Petkova's short exact sequence (Theorem \ref{petkova}).

\begin{maintheorem}\label{main-ses}
If $\dots\to P_2\to P_1$ is an ANR resolution of a coronated polyhedron $X$, then its Steenrod--Sitnikov
homology fits in a natural short exact sequence
\[0\to\derlim H_{n+1}(P_i)\to H_n(X)\to\lim H_n(P_i)\to 0.\]
\end{maintheorem}

The precise meaning of naturality is explained in the statement of Theorem \ref{main-ses2}, which also
notes that the proof of Theorem \ref{main-ses} works for any generalized homology theory%
\footnote{By a generalized (co)homology theory we mean, as usual, a theory which satisfies 
the Eilenberg--Steenrod axioms except for the dimension axiom.}
$H_*$ on closed pairs of metrizable spaces which satisfies the map excision axiom and the axiom of 
additivity with respect to metric wedges (see \cite{M00}*{\S\ref{book:map excision}, 
\S\ref{book:coadditivity}} concerning these axioms).
These are essentially the same two axioms as those used by Milnor (in addition to the 
Eilenberg--Steenrod axioms) to characterize Steenrod homology on compacta \cite{Mi1}.

\begin{example}[Comb and Flea] \label{comb1}
Let $P=\{\frac1n\mid n\in\N\}\x[0,1]\cup (0,1]\x\{1\}\subset\R^2$ 
and let $X=P\cup\{(0,0)\}$.
Since $P$ is a polyhedron, $X$ is a coronated polyhedron.
Each $R_i\bydef X\cup[0,\frac1n]\x\{0\}$ is locally contractible and hence an ANR.
Since every neighborhood of $X$ in $\R^2$ contains some $R_i$, the inclusion maps
$\dots\to R_2\to R_1$ form a resolution of $X$ (see \cite{M00}*{Lemma \ref{book:semi-resolution}(a)}).

The inverse sequence $\dots\to H_1(R_2)\to H_1(R_1)$ is of the form 
$\dots\emb\bigoplus_{i=2}^\infty\Z\emb\bigoplus_{i=1}^\infty\Z$.
Hence its $\lim^1$ is isomorphic to $\prod_{i=1}^\infty\Z/\bigoplus_{i=1}^\infty\Z$ 
(see \cite{M00}*{Example \ref{book:jacob-ladder}}).
Therefore $\tilde H_0(X)\simeq\prod_{i\in\N}\Z\big/\bigoplus_{i\in\N}\Z$ and $H_1(X)=0$.

It might at first seem that this computation contradicts the definition of $H_n(X)$
as $\colim H_n(K_\alpha)$ over all compact $K_\alpha\subset X$.
But it does not.
Let $K=\{\frac1n\mid n\in\N\}\x\{0\}\cup\{(0,0)\}$.
Then $K$ is a compact subset of $X$ and $\tilde H_0(K)\simeq\prod_{i=1}^\infty\Z$.
It is easy to see that the inclusion induced map $\tilde H_0(K)\to\tilde H_0(X)$ is surjective and
its kernel is the subgroup $\bigoplus_{i=1}^\infty\Z$ of $\prod_{i=1}^\infty\Z$.
\end{example}

\begin{example} \label{comb2}
Using the notation of the previous example, let us choose the basepoint $x$ at the origin.
It is easy to see that the inclusion-induced map of Steenrod--Sitnikov homotopy sets 
$\pi_0(K,x)\to\pi_0(X,x)\simeq\colim\pi_0(K_\alpha,x)$ (where $K_\alpha$ runs over all
compact subsets of $X$) is surjective and identifies the classes of 
the maps $S^0\xr{\cong}\{(\frac1n,0),(0,0)\}\subset K$ for all $n\in\N$, but does not 
identify them with the class of the constant map 
(see \cite{M00}*{Example \ref{book:ssh-ash}} for the details).
Thus $\pi_0(X)$ consists of two elements.

On the other hand, the inverse sequence $\dots\to\pi_1(R_2)\to\pi_1(R_1)$ is of the form 
$\dots\emb\join_{i=2}^\infty\Z\emb\join_{i=1}^\infty\Z$.
This does not satisfy the Mittag-Leffler condition and consists of countable groups,
so $\derlim\pi_1(R_i)$ is uncountable (see \cite{M00}*{Theorem \ref{book:gray}}).

This does not yet rule out a short exact sequence of pointed sets
\[0\to\derlim\pi_1(P_i)\to\pi_0(X)\to\lim\pi_0(P_i)\to 0,\]
because trivial kernel does not imply injectivity for pointed sets.
However let $Y=X\cup E$, where $E=\{-1\}\x[0,1]\cup\{0,1\}\x[-1,0]$, and let
$Q_i=P_i\cup E$.
Then $\pi_0(Y)$ is trivial and $\derlim\pi_1(Q_i)$ is uncountable, so we cannot possibly 
have a short exact sequence of pointed sets
\[0\to\derlim\pi_1(Q_i)\to\pi_0(Y)\to\lim\pi_0(Q_i)\to 0.\]
\end{example}

It seems crucial for the latter example that $\pi_0(X)$ is not a group.
Yet the proof of Theorem \ref{main-ses} does not work to give an
affirmative solution of the following problem, as it uses excision in an essential way.

\begin{problem} If $\dots\to P_2\to P_1$ is an ANR resolution of a coronated polyhedron $X$, 
does its Steenrod--Sitnikov homotopy fit in a short exact sequence
\[0\to\derlim\pi_{n+1}(P_i)\to\pi_n(X)\to\lim\pi_n(P_i)\to 0\]
for $n\ge 1$?
\end{problem}

S. V. Petkova (refining previous work of Sklyarenko \cite{Sk71}) proved that Steenrod--Sitnikov homology and 
\v Cech cohomology are unique ordinary homology theories on the category of closed pairs of 
locally compact separable metrizable spaces that satisfy map excision, $\bigsqcup$-additivity 
and $\nullseq$-additivity \cite{Pe2} (homology), \cite{Pe1}*{Theorem 9} (cohomology).
The same approach works without assuming the dimension axiom, except that ``uniqueness'' is understood
in the sense that any natural transformation into another such theory is an isomorphism as long as it is 
an isomorphism on $pt$ (see \cite{M00}*{Theorem \ref{book:uniqueness2}(b)}).

As a byproduct of the proof of Theorem \ref{main-ses} we obtain a dual theorem.

\begin{maintheorem} \label{uniqueness2}
Let $h_*$ (resp.\ $h^*$) be a generalized (co)homology theory on closed pairs of coronated polyhedra
satisfying map excision, $\bigsqcup$-additivity and $\nullseq$-additivity.
Then any natural transformation of $h_*$ (resp.\ $h^*$) into another such theory is an isomorphism as long as 
it is an isomorphism on $pt$.
\end{maintheorem}

%

As a byproduct of the proof of Theorem \ref{main-ses} we also obtain

\begin{maintheorem} \label{invariance} Steenrod--Sitnikov homology is an invariant of strong shape for coronated polyhedra.
\end{maintheorem}

E. G. Sklyarenko asked if Steenrod--Sitnikov homology is an invariant of strong shape for paracompact spaces, 
adding that the solution of this problem ``will likely be a test not only for Steenrod--Sitnikov homology, 
but also for strong shape theory itself'' \cite{Sk6}.

\begin{problem} \label{invariance0} Are Steenrod--Sitnikov homotopy groups invariant under strong shape for coronated polyhedra?
\end{problem}

\section{More examples} \label{examples}

\subsection{Complements of local compacta}

\begin{lemma} \label{open-in-closure}
Every locally compact subspace $X$ of a metrizable space $M$ is open in its closure.
\end{lemma}

This is true also when $M$ is non-metrizable (see \cite{En}*{3.3.9}).
Clearly, when $M$ is locally compact, the converse holds: if $X$ is open in its closure, 
then it is locally compact.

\begin{proof}
Suppose that some $x\in X$ is the limit of a sequence of points $x_k\in\bar X\but X$.
We may assume without loss of generality that each $x_k$ lies in an open neighborhood $U$ of $x$ in $\bar X$
such that $U\cap X$ has compact closure $K$ in $X$.
Each $x_k$ is the limit of a sequence of points of $X$, hence the limit of a sequence of points of $U\cap X$, 
which are also points of $K$.
But then $x_k\in K$, which is a contradiction.
\end{proof}

\begin{theorem} \label{co-locally-compact} 
(a) Let $Q$ be a compact polyhedron. If $X\subset Q$ is locally compact, then $Q\but X$ is a coronated polyhedron.

(b) Let $Q$ be a compact ANR. If $X\subset Q$ is locally compact, then $Q\but X$ is a coronated ANR.
\end{theorem}

In fact, the proof shows that $Q\but X=K\cup P$, where $K$ is a compactum and $P=X\but K$ is 
a locally compact polyhedron (in (a)) or a locally compact ANR (in (b)).

\begin{proof} 
Since $X$ is locally compact, it is open in its closure $\bar X$ in $Q$.
Hence $K\bydef \bar X\but X$ is closed in $\bar X$ and therefore in $Q$.
Since $S^n$ is compact, so is $K$.
On the other hand, $P\bydef S^n\but\bar X$ is an open subset of $Q$, and hence it is locally compact.
We have $Q\but X=P\cup K$.
If $Q$ is an ANR, then so is $P$ (see \cite{M00}*{Lemma \ref{book:ANR-basic}(c)}), and hence
$Q\but X$ is a coronated ANR.
If $Q$ is a polyhedron, then so is $P$ \cite{AH}*{\S III.3, Satz II} (see also 
\cite{Sp}*{Hint to Exercise 3.A.3}, \cite{M00}*{Theorem \ref{book:open-subset}}), and hence 
$Q\but X$ is a coronated polyhedron.
\end{proof}

\subsection{Complements of coronated polyhedra and coronated ANRs}

The converse to Theorem \ref{co-locally-compact} is not true (for instance, $S^1\but S^0$ is 
a coronated polyhedron, but $S^2\but (S^1\but S^0)$ is not locally compact), but some related 
assertions do hold.

\begin{proposition} \label{cor-lc}
(a) If $Q$ is a compact ANR and $X\subset Q$ is of the form $X=K\cup P$, where 
$K$ is a compactum and $P=X\but K$ is open in $Q$ (hence an ANR), then $Q\but X$ is locally compact.

(b) Let $Q$ be a compact affine polyhedron in some $\R^n$, and suppose that $X$ is a subset of $Q$
of the form $X=K\cup P$, where $K$ is a compactum and $P=X\but K$ is a locally compact affine polyhedron.
Then there exists a coronated polyhedron $Y\subset Q$ such that $Y$ deformation retracts onto $X$
and $Q\but Y$ is locally compact.

(c) If $X$ is an $n$-dimensional coronated polyhedron of the form $X=K\cup P$, where $K$ is a compactum
and $P$ is a locally compact separable polyhedron, then $X$ admits an embedding in $S^{2n+1}$ 
(regarded as an affine polyhedron in $\R^{2n+2}$) such that $P$ is embedded onto an affine polyhedron.
\end{proposition}

\begin{proof}[Proof. (a)] $Q\but X=(Q\but K)\cap (Q\but P)$ is an open subset of the compactum $Q\but P$.
\end{proof}

\begin{proof}[(b)] Since $P$ is locally compact, it is open in its closure $\bar P$ in $Q$.
Hence $L\bydef \bar P\but P$ is compact and therefore $U\bydef Q\but L$ is an affine polyhedron.
Since $P$ is closed in $U$, it is triangulated by a subcomplex $P_0$ of some simplicial complex $U_0$ 
triangulating $U$ \cite{M00}*{Theorem \ref{book:subcomplex0}}.
Let $N$ be the second derived neighborhood of $P_0$ in $U_0$ and let $V$ be the interior of $|N|$ in $U$.
Then $V$ deformation retracts onto $P$, and extending by $\id_K$ we get a deformation retraction 
of $Y\bydef K\cup V$ onto $X$.
On the other hand, $V$ is open in $Q$, and so $Q\but Y$ is locally compact by (a).
\end{proof}

\begin{proof}[(c)] Let $g\:X\to S^{2n+1}$ be some embedding.
Since $g(P)$ is locally compact, it is open in its closure $\overline{g(P)}$ in $S^{2n+1}$, and hence 
$L\bydef \overline{g(P)}\but g(P)$ is compact.
Then $U\bydef S^{2n+1}\but L$ is an affine polyhedron, and by general position the proper embedding 
$g|_P\:P\to U$ is $\eps$-close to a proper PL embedding $h\:P\to U$, where $\eps\:P\to (0,\infty)$ 
is given by $\eps(x)=d\big(g(x),\,L)$.
Clearly, $g|_K\cup h$ is an embedding of $X$ in $S^{2n+1}$.
\end{proof}

\subsection{Pinched mapping telescope}

\begin{example} \label{telescope-quotient}
Let $\dots\xr{p_1}P_1\xr{p_0}P_0$ be an inverse sequence of ANRs and continuous maps which is
a {\it semi-resolution}, that is, for each $n$ and every neighborhood $U$ of $p^\infty_n(P_\infty)$ in 
$P_n$ there exists an $m\ge n$ such that $p^m_n(P_n)\subset U$.
See \cite{M00} for a discussion of this definition.

The extended mapping telescope $P_{[0,\infty]}$ is the inverse limit of the finite metric telescopes
$P_{[0,n]}=MC(p_0)\cup_{P_1} MC(p_1)\cup_{P_2}\dots\cup_{P_{n-1}} MC(p_{n-1})$ (see details in
\cite{M00}*{\S\ref{book:extended}}).
We have $P_{[0,\infty]}=P_{[0,\infty)}\cup P_\infty$, where $P_\infty=\lim P_i$ and the infinite mapping 
telescope $P_{[0,\infty)}$ is an ANR (see \cite{M00}*{Proposition \ref{book:extended-ANR}}).
Since the inverse sequence is a semi-resolution, the metric quotient $P_{[0,\infty]}/P_\infty$ is 
well-defined and in fact is homeomorphic to the inverse limit of the metric quotients
$P_{[0,n+1]}/P_{n+1}=P_{[0,n]}\cup_{P_n}CP_n$ (see \cite{M00}*{\S\ref{book:pinched}}).

We will write $P_{[0,\infty)}^\Star\bydef P_{[0,\infty]}/P_\infty$ and $\STAR\bydef \{P_\infty\}$, where $P_\infty$
is regarded as a point of $P_{[0,\infty)}^\Star$.
More generally for any $J\subset [0,\infty)$ we denote by $P_J^\Star$ the subset $P_J\cup\STAR$ of 
$P_{[0,\infty)}^\Star$.

By the above $P_{[0,\infty)}^\Star$ is a coronated ANR.
Clearly, its subspace $P_{\N^+}^\Star$ is nothing but the null-sequence $\nullseq_{i\in\N} P_i$.

Let us also note that if $P_\infty$ is not compact, then $P_{[0,\infty)}^\Star$ is not locally compact.
Indeed, if $P_{[0,\infty)}^\Star$ is locally compact, then there exists a $k$ such that the $P_i$ are 
compact for all $i>k$, and hence $P_\infty$ is compact.
\end{example}

\begin{example} \label{telescope-quotient2}
Let $\dots\xr{p_1}P_1\xr{p_0}P_0$ be a semi-resolution consisting of polyhedra and polyhedral maps.
Then in the notation of Example \ref{telescope-quotient}, $P_{[0,\infty)}$ is a polyhedron 
\cite{M00}*{Theorem \ref{book:extended telescope}(a)}.
So $P_{[0,\infty)}^\Star$ is a coronated polyhedron.
\end{example}

\section{Coronated ANRs} \label{ANR-section}

\subsection{Detecting a coronated ANR}

\begin{lemma} \label{emb-coronated-anr} 
Let $M$ be a metrizable space and $\dots\subset X_2\subset X_1\subset M$.
Let $X_\infty=\bigcap_k X_k$.

If each $X_k$ is an ANR and each neighborhood of $X_\infty$ contains some $X_k$, then $X_\infty$ contains 
a compact subset $Z$ such that $X_\infty\but Z$ is an ANR.
\end{lemma}

\begin{proof}
Let $N_k$ be the interior of $X_\infty$ in $X_k$.
Let us note that each $N_k\subset N_{k+1}$ due to $X_{k+1}\subset X_k$.
Since $N_k$ is open in $X_k$, it is an ANR (see \cite{M00}*{\ref{book:ANR-basic}(c)}).
Also $N_k$ is open in $X_\infty$ and moreover in $N\bydef \bigcup_k N_k$.
Hence $N$ as a space has open neighborhoods of points that are ANRs and hence is an ANR itself
(see \cite{M00}*{Theorem \ref{book:ANR-union4}}).
Since each $N_k$ is open in $X_\infty$, so is $N$.
Then $Z\bydef X_\infty\but N$ is closed in $X_\infty$.
By definition, $Z$ is the set of all $x\in X_\infty$ such that each neighborhood of $x$ in each $X_k$ 
is not contained in $X_\infty$.

Suppose that $Z$ is not compact.
Let $z_1,z_2,\dots$ be a sequence of points of $Z$ that has no cluster points in $Z$.
Then it also has no cluster points in $X_\infty$ (since $Z$ is closed in $X_\infty$).
Since each $z_n\in Z$, the open $\frac1n$-ball about $z_n$ contains 
a point $x_n\in X_n\but X_\infty$. 
If $x$ is a cluster point of the sequence $x_1,x_2,\dots$, then it is also a cluster point of $z_1,z_2,\dots$,
and therefore lies in $M\but X_\infty$.
Thus the closure $Y$ of $\{x_1,x_2,\dots\}$ is contained in $M\but X_\infty$.
Then $M\but Y$ is an open neighborhood of $X_\infty$, so by our hypothesis it contains some $X_k$.
Then $x_k\notin X_k$, which contradicts our choice of $x_k$.
Thus $Z$ is compact.
\end{proof}

\begin{lemma} \label{B2-application}
Let $X_\infty$ be a metrizable space and $\dots\xr{p^2_1}X_1\xr{p^1_0}X_0$ be 
its sequential resolution; in fact, it is enough to assume (B2).
Let $Y=\{y_1,y_2,\dots\}$ be a countable subset of $X_\infty$ that has no cluster points in $X_\infty$.

Then there exists a $k$ such that $Z\bydef p^\infty_k(\{y_k,y_{k+1},\dots,\})$ has no cluster points in $X_k$
and there exists a neighborhood $M$ of $Z$ in $p^\infty_k(X_\infty)$ such that $p^\infty_k$ 
embeds $(p^\infty_k)^{-1}(M)$.
\end{lemma}

\begin{proof} Clearly, each $Y_i\bydef Y\but\{y_i\}$ is closed in $X_\infty$.
Hence $C\bydef \{X_\infty\but Y_i\mid i\in\N\}$ is an open cover of $X_\infty$.
Then by (B2) there exists an $n$ and an open cover $D$ of $X_n$ such that $(p^\infty_n)^{-1}(D)$ refines $C$.
Let $Y'=p^\infty_n(Y)$ and $y'_i=p^\infty_n(y_i)$.
If $i\ne j$, no element of $C$ contains both $y_i$ and $y_j$, and hence no element of $D$ contains both 
$y'_i$ and $y'_j$.
In particular, $y'_i\ne y'_j$ whenever $i\ne j$, and $Y'$ has no cluster points in $X_n$.
Clearly, each $Y'_i\bydef Y'\but\{y'_i\}$ is closed in $X_n$.
Let $\rho_i=\min(\frac1i,\frac{d_i}3)$, where $d_i=d(y'_i,Y'_i)$.
Let $B'_i$ be the closed $\rho_i$-ball about $y'_i$ in $p^\infty_n(X_\infty)$.
Then $B'_i\cap B'_j=\emptyset$ whenever $i\ne j$, and if $z_j\in B'_{i_j}$, where $i_1<i_1<\dots$, then
$d(z_j,y'_{i_j})\to 0$ and consequently the sequence $z_j$ has no cluster points in $X_n$.

Let $B_i=(p^\infty_n)^{-1}(B_i')$.
Thus $B_i$ is a closed neighborhood of $y_i$ and $p^\infty_n(B_i)=B_i'$.
Let $m_i=\sup\,\{m\mid p^\infty_m\text{ does not embed }B_i\}\in\N\cup\{\infty\}$.
Suppose that some subsequence $m_{i_j}\to\infty$, where $i_1<i_2<\dots$ and $\N\cup\{\infty\}$ has the topology 
of the one-point compactification of $\N$.
Let us choose some $m'_j\in\N$ such that $m'_j\le m_{i_j}$ and $m'_j\to\infty$.
Then $f_j\bydef p^\infty_{m'_j}|_{B_{i_j}}$ is not an embedding for each $j$.
If $f_j$ is not injective, then we have distinct points $a,b\in B_{i_j}$ such that $f_j(a)=f_j(b)$.
In this case let us set $Q_j=\{a\}$ and $r_j=b$. 
If $f_j$ is injective, but the inverse map $f_j^{-1}\:f_j(B_{i_j})\to B_{i_j}$ is not continuous, then 
there exists a sequence of points $a_k\in B_{i_j}$ and a point $b\in B_{i_j}$ such that $a_k\not\to b$ but
$f_j(a_k)\to f_j(b)$.
By passing to a subsequence we may assume that $b$ is not a cluster point of the sequence $a_k$.
In this case let $Q_j$ be the closure of $\{a_1,a_2,\dots\}$ in $X_\infty$ and let $r_j=b$.
Thus in both cases $r_j\in B_{i_j}$ and $Q_j$ is a closed subset of $B_{i_j}$ which does not contain $r_j$,
but the closure of $f_j(Q_j)$ contains $f_j(r_j)$.
The latter implies that the closure of $p^\infty_m(Q_j)$ contains $p^\infty_m(r_j)$ for all $m\le m'_j$.

Let us show that $Q\bydef \bigcup_j Q_j$ is closed in $X_\infty$.
Given a sequence of points $q_k\in Q$ converging to a point $q\in X_\infty$, their images 
$q'_k\bydef p^\infty_n(q_k)$ in $X_n$ converge to $q'\bydef p^\infty_n(q)$.
If the $q_k$ all lie in a finite union of $Q_j$'s, which is certainly closed, then $q\in Q$.
Otherwise there exists a subsequence $q_{k_l}\in Q_{j_l}$, where $j_1<j_2<\dots$.
Then each $q'_{k_l}\in B'_{i_{j_l}}$, where $i_{j_1}<i_{j_2}<\dots$ due to $j_1<j_2<\dots$.
Then by the above the sequence $z_l\bydef q'_{k_l}$ has no limit, which is a contradiction. 
Thus $Q$ is closed in $X_\infty$, and similarly $R\bydef \{r_1,r_2,\dots\}$ is closed in $X_\infty$.
Then $\tilde C\bydef \{X_\infty\but Q,\,X_\infty\but R\}$ is an open cover of $X_\infty$.
Hence by (B2) there exists an $m$ and an open cover $\tilde D$ of $X_m$ such that $(p^\infty_m)^{-1}(\tilde D)$
refines $\tilde C$.
Since $m'_j\to\infty$, there exists a $j$ such that $m'_j\ge m$.
Then the closure of $p^\infty_m(Q_j)$ contains $p^\infty_m(r_j)$.
If $U$ is an element of $\tilde D$ containing $p^\infty_m(r_j)$, then $(p^\infty_m)^{-1}(U)$ contains both $r_j$ 
and some point of $Q_j$, and hence does not lie in any element of $\tilde C$.
This is a contradiction.
Thus there exist an $m$ and an $l$ such that $m_i<m$ for all $i\ge l$.  

We have proved that $p^\infty_m$ embeds $B_i$ for all $i\ge l$.
Let $k=\max(l,m,n)$.
Then $p^\infty_k$ embeds $B\bydef \bigcup_{i\ge k} B_i$.
We have $B=(p^\infty_n)^{-1}(B')$, where $B'=\bigcup_{i\ge k} B'_i$.
Each $B'_i$ is a neighborhood of $y'_i$ in $p^\infty_n(X_\infty)$, so $B'$ is a neighborhood 
of $Z'\bydef \{y'_k,y'_{k+1},\dots\}$ in $p^\infty_n(X_\infty)$.
If $\pi^k_n\:p^\infty_k(X_\infty)\to p^\infty_n(X_\infty)$ denotes the restriction of $p^k_n$, 
then $M\bydef (\pi^k_n)^{-1}(B')$ is a neighborhood of $Z\bydef p^\infty_k(\{y_k,y_{k+1},\dots\})$ in $p^\infty_k(X_\infty)$ 
and $B=(p^\infty_k)^{-1}(M)$.
Finally, $Z$ has no cluster points in $X_k$ since $Z'$ has no cluster points in $X_n$.
\end{proof}

\begin{theorem} \label{coronated-anr} Let $X_\infty$ be a metrizable space and $\dots\xr{p^2_1}X_1\xr{p^1_0}X_0$ be 
its sequential resolution, where each $X_i$ is an ANR and moreover each $p^i_j(X_i)$ is an ANR.
Then $X_\infty$ contains a compact subset $Z$ such that $X_\infty\but Z$ is an ANR.
\end{theorem}

\begin{proof}
Let $N_k$ be the set of all $x\in X_\infty$ such that there exists an $i\ge k$ and an open neighborhood $V$ of 
$p^\infty_k(x)$ in $p^i_k(X_i)$ such that the restriction of $p^\infty_k$ to $U\bydef (p^\infty_k)^{-1}(V)$ is 
a homeomorphism between $U$ and $V$.
Since $p^\infty_k\:X_\infty\to p^i_k(X_i)$ is continuous, $U$ is open in $X_\infty$, and it follows that $N_k$ 
is open in $X_\infty$.
Clearly $p^\infty_k$ restricts to a homeomorphism between $N_k$ and $N'_k\bydef p^\infty_k(N_k)$, and clearly
$(p^\infty_k)^{-1}(N'_k)=N_k$.
Also for each $x\in N'_k$ there exists an $i\ge k$ such that some open neighborhood $V$ of $x$ in $N'_k$ 
is open in $p^i_k(X_i)$.
Since $V$ is open in $p^i_k(X_i)$, it is an ANR (see \cite{M00}*{Lemma \ref{book:ANR-basic}(c)}).
Then $N'_k$ is an ANR (see \cite{M00}*{Theorem \ref{book:ANR-union4}}), and hence so is $N_k$.
Since $N_k$ is open in $X_\infty$, it is also open in $N\bydef \bigcup_k N_k$.
Hence $N$ is an ANR (see \cite{M00}*{Lemma \ref{book:ANR-union3}(b)}).
Since each $N_k$ is open in $X_\infty$, so is $N$.
Then $Z\bydef X_\infty\but N$ is closed in $X_\infty$.

Suppose that $Z$ is not compact.
Let $y_1,y_2,\dots$ be a sequence of points of $Z$ that has no cluster points in $Z$.
Then it also has no cluster points in $X_\infty$ (since $Z$ is closed in $X_\infty$).
By Lemma \ref{B2-application} there exists a $k$ such that, writing $y'_i=p^\infty_k(y_i)$, the set 
$Y'\bydef \{y'_k,y'_{k+1},\dots\}$ has no cluster points in $X_k$ and there exists a closed neighborhood $M$ of $Y'$ 
in $p^\infty_k(X_\infty)$ such that $p^\infty_k$ embeds $(p^\infty_k)^{-1}(M)$.
Since $y_i\notin N_k$, every neighborhood of $y'_i$ in $M$ is not open in $p^i_k(X_i)$ for each $i\ge k$.
Hence for each $i\ge k$ the $\frac1i$-neighborhood of $y'_i$ in $p^i_k(X_i)$ contains a point 
$w_i\notin p^\infty_k(X_\infty)$.
If $w$ is a cluster point of $W\bydef \{w_k,w_{k+1},\dots\}$, then it is also a cluster point of $\{y'_k,y'_{k+1},\dots\}$,
but the latter has no cluster points by Lemma \ref{B2-application}.
Thus $W$ is closed in $X_k$.
Hence $X_k\but W$ is an open neighborhood of $p^\infty_k(X_\infty)$, so by (B1) it contains $p^i_k(X_i)$ 
for some $i\ge k$.
Then $w_i\in p^i_k(X_i)\subset X_k\but W$, which contradicts the definition of $W$.
Thus $Z$ is compact.
\end{proof}

\subsection{Constructing a sequential resolution}

\begin{theorem} \label{coronated-anr2} 
Let $X_\infty$ be a metrizable space containing a compact subset $Z$ such that $X_\infty\but Z$ is an ANR.
Then $X_\infty$ has a resolution of the form $\dots\subset X_2\subset X_1$, where each $X_i$ is an ANR.
\end{theorem}

\begin{proof} Let us embed $Z$ is some absolute retract $M$ (see \cite{M00}*{Corollary \ref{book:w-embedding}}).
Then the inclusion $Z\to M$ extends to a continuous map $f\:X_\infty\to M$.
Let $M_i$ be the open $\frac1i$-neighborhood of $Z$ in $M$.
Let $U_i$ be the open $\frac1i$-neighborhood of $Z$ in $f^{-1}(M_i)$.
Since $U_i$ is open $f^{-1}(M_i)$, it is also open in $X_\infty$.
Hence $U_i'\bydef U_i\but Z$ is open in $X_\infty'\bydef X_\infty\but Z$, so $U'_i$ is an ANR (see 
\cite{M00}*{Lemma \ref{book:ANR-basic}(c)}).

Let $A_1$ be the metric mapping cylinder of $f|_{U_i}\:U_i\to M_i$ relative to $Z$; that is, $A_1$ is 
the adjunction space $MC(f|_{U_i})\cup_\phi M_i$, where $\phi\:MC(f|_Z)\to M_i$ is the natural map onto 
the image.
Since $Z$ is compact, $f|_Z$ is perfect, and hence by \cite{M00}*{Corollary \ref{book:MC-ANR}} $A_1$ is an ANR.
Let us note that $A_1\but M_i=U_i'\x (0,1]$.
Let $A_2$ be the subset $U_i'\x (0,1]\cup X_\infty'\x\{1\}$ of the product $X_\infty'\x (0,1]$.
A homeomorphic copy of $A_2$ is homotopy dense in the metric mapping cylinder $MC(j)$ of 
the inclusion map $j\:U_i'\to X_\infty'$, whence $A_2$ is also an ANR 
(see \cite{M00}*{Corollary \ref{book:MC-ANR} and Theorem \ref{book:homotopy dense}}).
Let $X_i$ be the amalgamated union $A_1\cup_{U_i'\x(0,1]} A_2$ (with the quotient topology).
Let us show that $X_i$ is metrizable.
Let $V_1$ and $V_2$ be disjoint open neighborhoods of $X_\infty\but U_i$ and $Z$ in $X_\infty$.
Let us consider $F_1\bydef A_1\but\big((V_1\cap U_i)\x(\frac23,1]\big)$ and 
$F_2\bydef (X_\infty\but V_2)\x\{1\}\cup (U_i\but V_2)\x [\frac13,1]$.
Then $F_1$ and $F_2$ are closed in $X_i$ and $F_1\cup F_2=X_i$.
Each $F_k$, being a subset of $A_k$, is metrizable.
Hence $X_i$ is metrizable (see \cite{M00}*{Corollary \ref{book:amalgam-metrization2}(a)}).
Since $A_1$ and $A_2$ are ANRs which are open in $X_i$, we obtain that $X_i$ is an ANR 
(see \cite{M00}*{Corollary \ref{book:ANR-union2}(a)}).

We have $\dots\subset X_1\subset X_0$ and $X_\infty=\bigcap_i X_i$. 
If $U$ is an open neighborhood of $X_\infty$ in $X_0$, then $d(X_0\but U,\,Z)>0$ and hence $U$ contains 
some $X_i$. 
Thus $\dots\subset X_1\subset X_0$ is a resolution of $X_\infty$ 
(see \cite{M00}*{Lemma \ref{book:semi-resolution}}).
\end{proof}

\begin{corollary}\label{ANR-main} Let $X_\infty$ be a metrizable space.
The following are equivalent:
\begin{enumerate}
\item $X_\infty$ has a resolution of the form $\dots\subset X_2\subset X_1$, where each $X_i$ is an ANR;
\item $X_\infty$ has a resolution of the form $\dots\xr{p^3_2} X_2\xr{p^2_1}X_1$, where each $p^i_j(X_j)$ 
is an ANR;
\item $X_\infty$ contains a compact subset $Z$ such that $X_\infty\but Z$ is an ANR.
\end{enumerate}
\end{corollary}

\begin{proof} Clearly, (1)$\Rightarrow$(2). 
By Theorem \ref{coronated-anr} (2)$\Rightarrow$(3).
By Theorem \ref{coronated-anr2} (3)$\Rightarrow$(1).
Also, a direct proof of (1)$\Rightarrow$(3) is given by Lemma \ref{emb-coronated-anr}.
\end{proof}

\section{Coronated polyhedra} \label{poly-section}

\subsection{Detecting a coronated polyhedron}

\begin{theorem} \label{coronated-poly} Let $X_\infty$ be a metrizable space and $\dots\xr{p^2_1}X_1\xr{p^1_0}X_0$ 
be its sequential resolution, where each $X_i$ is an affine polyhedron and each $p^{i+1}_i$ is a polyhedral map.
Then $X_\infty$ contains a compact subset $Z$ such that $X_\infty\but Z$ is a polyhedron.
\end{theorem}

\begin{proof} By Theorem \ref{coronated-anr} $X_\infty$ contains a compact subset $Z$ such that 
$N\bydef X_\infty\but Z$ is an ANR.
We will show that $N$ is a polyhedron.
Let $N_k$ and $N'_k$ be as in the proof of Theorem \ref{coronated-anr}.
Then for each $x\in N'_k$ there exists an $i\ge k$ such that some open neighborhood $V$ of $x$ in $N'_k$ 
is open in $p^i_k(X_i)$.
Since $p^i_k(X_i)$ is an affine polyhedron, so is $V$ \cite{M00}*{Corollary \ref{book:opensubset2}}.
Then $N'_k$ is an affine polyhedron \cite{M00}*{Theorem \ref{book:affine manifold}}.
It is easy to see that each $N_k\subset N_{k+1}$.
Since $N_k$ is open in $X_\infty$, it is open in $N_{k+1}$.
Then $N''_k\bydef (p^{k+1}_k)^{-1}(N_k)$ is an open subset of $N_{k+1}$.
Hence $N''_k$ is an affine polyhedron \cite{M00}*{Theorem \ref{book:affine manifold}}.
Since $p^{k+1}_k$ is polyhedral, and its restriction $p^{k+1}_k|_{N''_k}\:N''_k\to N'_k$ is a homeomorphism,
this restriction is a polyhedral homeomorphism \cite{M00}*{Lemma \ref{book:homeo-restriction}}.
Thus each $N_k$ is endowed with a structure of a polyhedron, and the transition maps of these structures
are polyhedral.
Since $N_k$ is open in $X_\infty$, it is also open in $N\bydef \bigcup_k N_k$.
Hence $N$ is a polyhedron \cite{M00}*{Lemma \ref{book:atlas5}}.
\end{proof}

\subsection{Compactohedral inverse sequences}

Let $R=(\dots\xr{p_2}R_2\xr{p_1}R_1)$ be an inverse sequence.
We call it {\it compactohedral} if the following conditions hold:
\begin{enumerate}
\item[(C0)] each $R_i$ is a polyhedron and each $p_i$ is a polyhedral map;
\item[(C1)] each $R_i$ contains a compact subpolyhedron $K_i$, and each $p_i(K_{i+1})\subset K_i$;
\item[(C2)] each $K_{i+1}\subset\Int p_i^{-1}(K_i)$;
\item[(C3)] each $p_i$ restricts to a homeomorphism between $p_i^{-1}(R_i\but K_i)$ and $R_i\but K_i$.
\end{enumerate}
We call $R$ {\it weakly compactohedral} if it satisfies (C0), (C1) and (C3).
We call $R$ {\it weakly pre-compactohedral} if (C0), (C1) and the following conditions hold:
\begin{enumerate}
\item[(C2$'$)] each $R_i$ contains a closed supolyhedron $L_i$, and each $p_i(L_{i+1})\subset K_i\subset L_i$;
\item[(C3$'$)] each $p_i$ restricts to a homeomorphism between $p_i^{-1}(R_i\but L_i)$ and $R_i\but L_i$.
\end{enumerate}
We call $R$ {\it pre-compactohedral} if (C0), (C1) and the following conditions hold:
\begin{enumerate}
\item[(C2$''$)] each $R_i$ contains a closed supolyhedron $L_i$, and each $p_i(L_{i+1})\subset K_i\subset\Int L_i$;
\item[(C3$''$)] each $p_i$ restricts to a homeomorphism between $p_i^{-1}(\overline{R_i\but L_i})$
and $\overline{R_i\but L_i}$.
\end{enumerate}
We will call the $K_i$ and the $L_i$ the {\it associated compact and closed subpolyhedra} of the (weakly) 
[pre-]compactohedral inverse sequence.

If $R$ is (weakly) compactohedral, then it is (weakly) pre-compactohedral by setting $L_{i+1}=p_i^{-1}(K_i)$.

\begin{lemma} \label{pre-compactohedral}
If $R$ is (weakly) pre-compactohedral, then each $p_i$ factors into a composition of polyhedral maps 
$R_{i+1}\to R_i'\to R_i$ so that the inverse sequence $\dots\to R'_2\to R'_1$ is (weakly) compactohedral.
\end{lemma}

Let us note that the new inverse sequence has the same limit as $R$.
 
\begin{proof} We may assume that each $R_i$ is triangulated by a simplicial complex $P_i$ which has an admissible 
simplicial subdivision $P_i'$ such that each $p_i\:P_{i+1}\to P_i'$ is simplicial and $L_i$, $K_i$ are 
triangulated by subcomplexes $Q_i$, $S_i$ of $P_i$.
Let $S'_i$ be the subcomplex of $P'_i$ subdividing $S_i$ and let $q_i\:Q_{i+1}\to S'_i$ be the conization of 
$p_i|_{Q_{i+1}}$ (see \cite{M00}*{\S\ref{book:polytopal complexes}}).
Let $R'_i=|P_{i+1}\cup_{q_i}S'_i|$ (see \cite{M00}*{\S\ref{book:mcn}}).
Then the assertion on factorization follows from \cite{M00}*{Lemma \ref{book:conical adjunction}(a)}.
\end{proof}

\begin{lemma} \label{pre-poly}
Let $\dots\xr{p_2}R_2\xr{p_1}R_1$ be a weakly pre-compactohedral inverse sequence and let $K_i$ and $L_i$ be 
its associated compact and closed subpolyhedra.
Let $X=\lim R_i$ and $K=\lim K_i$.
Then $X\but K$ is a polyhedron.
\end{lemma}

In particular, this implies that $X$ is a coronated polyhedron.

\begin{proof}
We may assume that the $R_i$ are affine polyhedra.
Since $K=\bigcap_i (p^\infty_i)^{-1}(K_i)$, we have $X\but K=\bigcup_i P_i$, where 
$P_i=(p^\infty_i)^{-1}(R_i\but K_i)$.
Let $Q_i=(p^\infty_i)^{-1}(R_i\but L_i)$.
Since $K_i\subset L_i$, we have $R_i\but L_i\subset R_i\but K_i$, whence $Q_i\subset P_i$.
Since $p_i(L_{i+1})\subset K_i$, we have $p_i^{-1}(R_i\but K_i)\subset R_{i+1}\but L_{i+1}$, whence
$P_i\subset Q_{i+1}$.
Consequently $\bigcup_i Q_i=\bigcup_i P_i=X\but K$.
By \cite{M00}*{Corollary \ref{book:opensubset2}} each $R_i\but L_i$ (with the induced topology) is 
an affine polyhedron, and so is each $p_i^{-1}(R_i\but L_i)$.
Then by \cite{M00}*{Lemma \ref{book:homeo-restriction}} each homeomorphism 
$p_i^{-1}(R_i\but L_i)\xr{p_i|}R_i\but L_i$ is polyhedral.
Thus each $Q_i$ is a polyhedron and moreover the inclusion maps $Q_i\subset Q_{i+1}$ are polyhedral.
Hence by \cite{M00}*{Lemma \ref{book:atlas5}} $X\but K$ is a polyhedron.
\end{proof}

\begin{remark} Not every resolution of a coronated polyhedron is a weakly pre-compactohedral inverse sequence.
Indeed, $\dots\subset[0,1+\frac13)\subset[0,1+\frac12)$ is a resolution of $[0,1]$ (see 
\cite{M00}*{Lemma \ref{book:semi-resolution}(a)}).
Since each $[0,1+\frac1i)$ is non-compact, it is easy see that this inverse sequence is not weakly pre-compactohedral.
\end{remark}

One of the usual proofs of the Lebesgue lemma works to show the following.

\begin{lemma}[\cite{M2}*{\ref{metr:lebesgue+}}] \label{lebesgue}
Let $X$ be a metric space and $K\subset X$ be compact.
Then for every open cover $C$ of $X$ there exists an open neighborhood $O$ of $K$ and a $\lambda>0$ such that every 
closed ball $B_\lambda(x)$, where $x\in O$, lies in some element of $C$.
\end{lemma}

\begin{theorem} \label{resolution theorem}
Every weakly pre-compactohedral inverse sequence is a resolution of its limit.
\end{theorem}

\begin{proof} Let $\dots\xr{p_2} R_1\xr{p_1} R_0$ be a weakly pre-compactohedral inverse sequence, and let 
$K_i$ and $L_i$ be its associated compact and closed subpolyhedra.
Let $X=\lim R_i$ and $K=\lim K_i$.

Let us verify condition (B1): For each $n$ and every neighborhood $U$ of $p^\infty_n(X)$ in $R_n$ 
there exists an $m\ge n$ such that $p^m_n(X_m)\subset U$.
Now $\dots\to K_1\to K_0$ is a resolution of $K$ (see \cite{M00}*{Lemma \ref{book:compact-resolution}}), 
in particular, it satisfies (B1); thus there exists an $m>n$ such that $p^{m-1}_n(K_{m-1})\subset U\cap K_n$.
On the other hand, by the definition of a weakly pre-compactohedral inverse sequence, 
$p^m_{m-1}(R_m)$ lies in $K_{m-1}\cup p^m_{m-1}(R_m\but L_m)$. 
Hence $p^m_n(R_m)$ lies in 
$p^{m-1}_n\big(K_{m-1}\cup p^m_{m-1}(R_m\but L_m)\big)=p^{m-1}_n(K_{m-1})\cup p^m_n(R_m\but L_m)$.
Here $p^{m-1}_n(K_{m-1})\subset U$ by the above.
Finally, $R_m\but L_m$ lies in $p^\infty_m(X)$, so $p^m_n(R_m\but L_m)$ lies in 
$p^m_n\big(p^\infty_m(X)\big)=p^\infty_n(X)$, and therefore also in $U$.

Let us verify condition (B2): For every open cover $C$ of $X$ there exists an open cover $D$ of some $R_n$ 
such that $(p^\infty_n)^{-1}(D)$ refines $C$.
Let us fix some metric $d_i$ on each $R_i$ such that $R_i$ has diameter $\le 1$, and let us fix the metric
$d\big((x_i),(y_i)\big)=\sup_{i\in\N}2^{-i}d_i(x_i,y_i)$ on their inverse limit $X\subset\prod_i R_i$.
By Lemma \ref{lebesgue} there exists an open neighborhood $O$ of $K$ and a $\lambda>0$ such that every 
closed ball $B_\lambda(x)$, where $x\in O$, lies in some element of $C$.
Let us choose $m$ so that $2^{-m+1}<\lambda$ and $O$ contains the $2^{-m}$-neighborhood of $K$.
Let $U_m$ be the $2^{-m-1}$-neighborhood of $f^\infty_m(K)$ in $K_m$.
Now $\dots\to K_1\to K_0$ is a resolution of $K$ (see \cite{M00}*{Lemma \ref{book:compact-resolution}}), 
in particular, it satisfies (B1); thus there exists an $n>m$ such that $p^{n-1}_m(K_{n-1})\subset U_m$.
Let $V=(p^\infty_{n-1})^{-1}(K_{n-1})$.
Then $p^\infty_m(V)\subset U_m$, and it follows that $V$ lies in the $2^{-m}$-neighborhood of $K$ in $X$, 
hence in $O$.

Let $D_K$ be the open cover of the $2^{-m}$-neighborhood of $L_n$ by all balls of radius $2^{-m}$ centered 
at points of $L_n$.
Let $D_P$ be the open cover of $R_n\but L_n$ by the sets $p^\infty_n(W)\but L_n$ 
(which are homeomorphic copies of the sets $W\but (p^\infty_n)^{-1}(L_n)$) for all $W\in C$.
Then $D\bydef D_K\cup D_P$ is an open cover of $R_n$.
Clearly, $(p^\infty_n)^{-1}(D_P)$ refines $C$.
Since $p^n_{n-1}(L_n)\subset K_{n-1}$, we have $L_n\subset (p^n_{n-1})^{-1}(K_{n-1})$ and consequently
$(p^\infty_n)^{-1}(L_n)\subset (p^\infty_{n-1})^{-1}(K_{n-1})=V$.
Thus if $x\in L_n$ and $y\in (p^\infty_n)^{-1}(x)$, then $y$ lies in $V$ and hence in $O$.
Since $n\ge m+1$, we have $(p^\infty_n)^{-1}\big(B_{2^{-m}}(x)\big)\subset B_{2^{-m+1}}(y)$.
Then each element of $(p^\infty_n)^{-1}(D_K)$ lies in the ball of radius $2^{-m+1}$ centered at some point of $O$.
Hence $(p^\infty_n)^{-1}(D_K)$ refines $C$.
Thus $(p^\infty_n)^{-1}(D)$ also refines $C$.
\end{proof}

\subsection{Constructing a sequential resolution}

\begin{theorem} \label{compactohedral}
Let $X$ be a metrizable space and $K$ a compact subset of $X$ such that $X\but K$ is a polyhedron.
Then $X$ is the limit of a compactohedral inverse sequence.
\end{theorem}
 
\begin{proof} We can represent $K$ as the limit of an inverse sequence $\dots\xr{q_1}|K_1|\xr{q_0}|K_0|$
where each $K_i$ is a finite simplicial complex and each $q_i$ is a simplicial map $K_{i+1}\to K_i'$, where
$K_i'$ is a subdivision of $K_i$. 
We may assume that $K_0=pt$.
Let $|K|_{[0,\infty]}=|K|_{[0,\infty)}\cup K$ be the compactified mapping telescope.
Let us fix some metrics on $X$ and on $|K|_{[0,\infty)}$.
There exists a homeomorphism $h\:|K|_{[0,\infty)}\to|K_{[0,\infty)}|$ 
(see \cite{M00}*{Theorem \ref{book:telescope-theorem}}); by construction, the diameters of the images in
$|K|_{[0,\infty]}$ of the simplexes of $K_{[0,\infty)}$ tend to $0$ as they approach $K$.

Let $P=X\but K$.
By \cite{M00}*{Lemma \ref{book:mesh-lemma}} there exists a homeomorphism $H\:P\to |L|$
for some affine simplicial complex $L$ such that the diameters of the images in $X$ of the simplexes of $L$ 
tend to $0$ as they approach $K$.
Let $N_\eps$ be the open $\eps$-neighborhood of $K$ in $X$, and let $\Lambda_n$ be the subcomplex of $L$
consisting of all simplexes that intersect $H(X\but N_{1/n})$ and all their faces.
Since $K_0=pt$, it follows that $\id_K$ extends to a map $f\:X\to |K|_{[0,\infty]}$ sending 
$P$ into $|K|_{[0,\infty)}$ (see \cite{M-I}*{Lemma \ref{fish:Milnor}(a)}).
Using the deformation retractions $|K|_{[0,\infty]}\to |K|_{[0,n]}$ 
(see \cite{M00}*{Proposition \ref{book:telescope retraction}}), we may assume that $f$ sends each 
$H^{-1}(|\Lambda_n|)$ into $|K|_{[0,n-1]}$.
By Sakai's theorem (see \cite{M00}*{Theorem \ref{book:simp-appr0}}) there exists an admissible subdivision $L'$
of $L$ such that the composition
$|L|\xr{H^{-1}}P\xr{f|_P}|K|_{[0,\infty)}\xr{h}|K_{[0,\infty)}|$ is $C$-homotopic to a simplicial map 
$\phi\:L'\to K_{[0,\infty)}$, where $C$ is the cover of $|K_{[0,\infty)}|$ by the open stars of vertices of
$K_{[0,\infty)}$.
Due to the condition on the diameters of simplexes, the composition
$P\xr{H}|L|\xr{\phi}|K_{[0,\infty)}|\xr{h^{-1}}|K|_{[0,\infty)}$
extends via $\id_K$ to a continuous map $g\:X\to |K|_{[0,\infty]}$.
On the other hand, clearly, $\phi$ sends each $|\Lambda_n|$ into $|K_{[0,n]}|$.
For any $J\subset [0,\infty)$ let $L_J$ denote the subcomplex of $L'$ such that $|L_J|=\phi^{-1}(|K_J|)$ and 
let $\phi_J$ denote the conization of the simplicial map $\phi|_{|L_J|}\:L_J\to K_J$ 
(see \cite{M00}*{\S\ref{book:polytopal complexes}}). 
We also write $L_{\{n\}}=L_n$ and $\phi_{\{n\}}=\phi_n$.
Thus we have $|\Lambda_n|\subset|L_{[0,n]}|$, and consequently
$X\but|L_{[0,n]}|\subset X\but|\Lambda_n|\subset N_{1/n}$.

Let $R_n=|L_{[0,n]}\cup_{\phi_n}K_n|$ (see \cite{M00}*{\S\ref{book:mcn}}).
Similarly let $R_{[n,n+1]}=|L_{[0,n+1]}\cup_{\phi_{[n,n+1]}}K_{[n,n+1]}|$.
Then \cite{M00}*{Lemma \ref{book:conical adjunction}(b)} yields a polyhedral map $R_{n+1}\to R_{[n,n+1]}$,
and the polyhedral retraction $|K_{[n,n+1]}|\to|K_n|$ yields a polyhedral map $R_{[n,n+1]}\to R_n$.
Let $p_n\:R_{n+1}\to R_n$ be their composition.
Let $R_{[n,\infty)}=|L_{[0,\infty)}\cup_{\phi_{[n,\infty)}}K_{[n,\infty)}|$ and let $r_n$ be 
the composition $X\to R_{[n,\infty)}\cup K_\infty\to R_n$.
It follows from \cite{M00}*{Lemma \ref{book:conical adjunction}(b)} that each composition 
$X\xr{r_{n+1}}R_{n+1}\xr{p_n}R_n$ equals $r_n$.

By construction $\dots\xr{p_1}R_1\xr{p_0}R_0$ is a compactohedral inverse sequence.
Let $R_\infty$ be its limit.
The maps $r_n$ yield a continuous map $r\:X\to R_\infty$ such that each composition 
$X\xr{r}R_\infty\xr{p^\infty_n}R_n$ equals $r_n$.
Clearly, $r$ is bijective, and $r|_P$ and $r|_K$ are homeomorphisms onto their images.
To show that $r^{-1}$ is continuous, it remains to show that a sequence of points $x_i\in P$ converges
to a point $x\in K$ if the sequence $r(x_i)$ converges to $r(x)$.

Suppose on the contrary that $r(x_i)\to r(x)$ but $x_i\not\to x$.
Since $p^\infty_nr=r_n$, we get that for each $n$ the sequence $r_n(x_i)$ converges to $r_n(x)$.
Suppose that there exists an $m$ such that $H^{-1}(|L_{[0,m]}|)$ contains an infinite subsequence $x_{i_j}$.
The image of $|L_{[0,m]}|$ in $R_{m+1}$ is a closed subset of $R_{m+1}$, disjoint from $K_{m+1}$.
Hence $r_n(x_i)\not\to r_n(x)$, which is a contradiction.
Thus we obtain that for each $m$ all except finitely many of the $x_i$ lie in $X\but|L_{[0,m]}|\subset N_{1/m}$. 
Thus $d(x_i,K)\to 0$ as $i\to\infty$.
Let us pick points $y_i\in K$ such that $d(x_i,y_i)=d(x_i,K)$.
Since $x_i\not\to x$, there exists an infinite subsequence $x_{i_j}$ such that $x$ is not its cluster point.
Since $K$ is compact, the sequence $y_{i_j}$ has a cluster point $y$.
Then $y$ is also a cluster point of the sequence $x_{i_j}$.
Since $r$ is continuous, $r(y)$ is a cluster point of $r(x_{i_j})$.
But we have assumed that $r(x_i)\to r(x)$, so $r(y)=r(x)$ and consequently $y=x$.
Thus $x$ is a cluster point of $x_{i_j}$, which is a contradiction.
\end{proof}

\begin{corollary} \label{poly-main} Let $X$ be a metrizable space.
The following are equivalent:
\begin{enumerate}
\item $X$ is the limit of a compactohedral inverse sequence;
\item $X$ is the limit of a weakly compactohedral inverse sequence;
\item $X$ is the limit of a pre-compactohedral inverse sequence;
\item $X$ is the limit of a weakly pre-compactohedral inverse sequence;
\item $X$ admits a resolution of the form $\dots\xr{p_2} R_2\xr{p_1}R_1$, where each $R_i$ is a polyhedron
and each $p_i$ is a polyhedral map;
\item $X$ contains a compact subset $K$ such that $X\but K$ is a polyhedron.
\end{enumerate}
\end{corollary}

\begin{proof}
Trivially, (1)$\Rightarrow$(2)$\Rightarrow$(4) and (1)$\Rightarrow$(3)$\Rightarrow$(4).

(4)$\Rightarrow$(5) by Theorem \ref{resolution theorem}.
(5)$\Rightarrow$(6) by Theorem \ref{coronated-poly}.
(6)$\Rightarrow$(1) by Theorem \ref{compactohedral}.

Also, Lemma \ref{pre-compactohedral} yields direct proofs of (3)$\Rightarrow$(1) and (4)$\Rightarrow$(2); and
Lemma \ref{pre-poly} yields a direct proof of (4)$\Rightarrow$(6).
\end{proof}

\subsection{Alternative construction}
If the directed set of open covers of a space $X$ contains a cofinal sequence $C_1,C_2,\dots$, then the inclusion 
of the inverse sequence $\dots\rightsquigarrow N_{C_2}\rightsquigarrow N_{C_1}$ of the homotopy classes of maps 
between the nerves into the \v Cech expansion of $X$ is an isomorphism in the pro-category 
(cf.\ e.g.\ \cite{MS}*{\S I.1.1, Theorem 1}).
Although it is not true that the directed set of open covers of a coronated polyhedron $X$ contains a cofinal sequence
(unless $X$ is UC-metrizable), the desired pro-categorical consequence still holds.

\begin{theorem} \label{cech-cofinal}
The \v Cech expansion of a coronated polyhedron $X$ contains an inverse sequence whose inclusion is an isomorphism
in the pro-category.
\end{theorem}



\begin{proof} Suppose that $K\subset X$ is a compactum and $P\bydef X\but K$ is a polyhedron.
Let us fix some metric on $X$ bounded above by $1$ and for any $A\subset X$ let $B_\eps(A)$ denote the cover of 
the closed $\eps$-neighborhood of $A$ by all closed balls $B_\eps(x)$, $x\in A$.
Let $C_0=\{X\}$, $U_0=O_0=X$ and $\lambda_0=1$.
Assume inductively that $C_n$ is a finite open cover of an open neighborhood $U_n$ of $K$ that is refined by 
$B_{\lambda_n}(O_n)$ for some open neighborhood $O_n$ of $K$ in $U_n$.
Let $C_{n+1}$ be a finite subset of $B_{\lambda_n}(O_n)$ that covers $K$; then it also covers some open 
neighborhood $U_{n+1}$ of $K$.
By Lemma \ref{lebesgue} there exists an open neighborhood $O_{n+1}$ of $K$ in $U_{n+1}$ such that 
$C_{n+1}$ is refined by $B_{\lambda_{n+1}}(O_{n+1})$ for some $\lambda_{n+1}>0$.
We may assume that $\lambda_{n+1}$ is so small that $\lambda_{n+1}\le\min(\lambda_n,\frac1{n+1})$ and
that $O_{n+1}$ contains the open $\lambda_{n+1}$-neighborhood of $X$.
But then we may assume that $O_{n+1}$ equals the open $\lambda_{n+1}$-neighborhood of $X$.
Since the metric on $X$ is bounded above by $1$, this condition also holds for $n+1=0$.
Thus we get open neighborhoods $X=U_0\supset O_0\supset U_1\supset O_1\supset\dots$ of $K$ and finite open covers 
$C_i$ of the $U_i$ such that each $C_n$ refines $B_{\lambda_{n-1}}(O_{n-1})$ and is refined by $B_{\lambda_n}(O_n)$, 
where each $O_n$ is the open $\lambda_n$-neighborhood of $X$ and each $\lambda_n\le\frac1n$.
In particular, each $C_{n+1}$ refines $C_n$.

Let $L$ be a triangulation of $P$ such that for each $n$, every simplex intersecting $O_n$ is 
of diameter $<\lambda_n-\lambda_{n+1}$.
Then for any vertex of $L$ that lies in $O_n$, its open star lies in some element of $C_n$; and a simplex of $L$
that intersects $P\but O_n$ does not intersect the closure of $O_{n+1}$.
Let $L_n$ be the subcomplex of $L$ consisting of all simplexes of $L$ that intersect $X\but O_n$ and of 
all their faces.
Then $|L_n|$ is a closed subset of the closed set $X\but O_{n+1}$, so it is closed in $X$.
Hence $V_i\bydef X\but |L_i|$ is an open neighborhood of $K$ in $O_i$.
Also $|L_n|$ lies in the interior of $X\but O_{n+1}$, which lies in $|L_{n+1}|$; thus $|L_n|\subset\Int |L_{n+1}|$. 
Since each $O_i$ lies in the $\frac1i$-neighborhood of $X$, we have $\bigcup |L_i|=P$.
Let $D_i$ be the set of all open stars of vertices of $L_i$ in $L$.
Then $E_i\bydef C_i|_{V_i}\cup D_i$ is an open cover of $X$.
Let us show that $E_{i+1}$ refines $E_i$.
Since $C_{i+1}$ refines $C_i$ and $V_{i+1}\subset V_i$, every element of $C_{i+1}|_{V_{i+1}}$ lies in some element 
of $C_i|_{V_i}$.
Every element of $D_i\subset D_{i+1}$ lies in itself, and an element of $D_{i+1}\but D_i$ is the open star of 
a vertex $v$ of $L$ that lies in $V_i$, and in particular in $O_i$.
Hence it lies in some element of $C_i$, and consequently also in some element of $C_i|_{V_i}$.

Finally, let $C$ be an open cover of $X$.
By Lemma \ref{lebesgue} there exists an open neighborhood $O$ of $K$ and a $\lambda>0$ such that every element of 
$B_\lambda(O)$ lies in some element of $C$.
Let $n$ be such that $O_{n-2}\subset O$ and $\lambda_{n-2}\le\lambda$.
Then every element of $C_{n-1}$ lies in some element of $C$.
Hence every element of $C_n|_{V_n}\cup (D_n\but D_{n-1})$ also lies in some element of $C$.
Let $L'$ be a subdivision of $L$ with new vertices only in $\Int |L_n|$ and such that the open star of every vertex
of $L_{n-1}'$ in $L'$ lies in come element of $C$.
(Hereafter $L_i'$ denotes the induced subdivision of $L_i$, which is a subcomplex of $L'$.)
Let $D'_i$ be the set of open stars of vertices of $L_i'$ in $L'$.
Then $D'_{n-1}$ refines $C$ and $D'_n\but D'_{n-1}$ refines $D_n\but D_{n-1}$.
Since $C_n|_{V_n}\cup (D_n\but D_{n-1})$ refines $C$, the open cover $E^C_n\bydef C_n|_{V_n}\cup D_n'$ of $X$ also
refines $C$.
On the other hand, since $(L',L_n')$ is a subdivision of $(L,L_n)$, the nerves of $D_n$ and $D_n'$ are homeomorphic.
Moreover, since the only new vertices of $L'$ are in $\Int L_n$, their open stars are disjoint from the elements
of $C_n|_{V_n}$, and it follows that $N_{E^C_n}$ is homeomorphic to $N_{E_n}$; let $h_C$ denote the homeomorphism.
Also, $E^C_n$ refines $E_n$, and the homotopy class of maps $N_{E^C_n}\to N_{E_n}$ given by this refinement clearly 
contains $h_C$.
Therefore this homotopy class is invertible, with inverse $[h_C]$.

Let us write $E_n=j(C)$ and $E^C_n=j'(C)$. 
Then $j$ is a map from the directed set $\Lambda$ of all open covers of $X$ to its totally ordered subset 
$E\bydef \{E_n\mid n\in\N\}$, and we can define a homotopy class $j_C\in [N_{j(C)},N_C]$ by composing 
$[h_C^{-1}]\in[N_{j(C)},N_{j'(C)}]$ with the homotopy class of maps $N_{j'(C)}\to N_C$ given by the refinement
of $C$ by $j'(C)=E^C_n$. 
Then $J\bydef (j,j_C)$ is an inv-morphism of the two inverse systems in the homotopy category.
Indeed, if $D$ refines $C$, then $j(D)=E_m$ for some $m$.
If, for example, $m\le n$, then it is clear that the following diagram commutes up to homotopy:
\[\begin{CD}
N_{E_m}@>h_D^{-1}>>N_{E^D_m}@>>>N_D\\
@AAA@AAA@|\\
N_{E_n}@>h_D^{-1}>>N_{E^D_n}@>>>N_D\\
@|@VVV@VVV\\
N_{E_n}@>h_C^{-1}>>N_{E^C_n}@>>>N_C.
\end{CD}\]
(For $J$ to be an inv-morphism, we need only the outer square of this diagram to be commutative.) 
The case $m\ge n$ is similar (and easier).

We also have the inclusion $i\:E\to\Lambda$ and the morphism $I\bydef (i,[\id])$ of the inverse systems.
Then $JI$ determines the same pro-morphism as the identity since it can be assumed that $j(E_n)=E_{n+2}$
and $j_{E_n}\in[N_{E_{n+2}},N_{E_n}]$ is given by refinement.
Finally, $IJ$ determines the same pro-morphism as the identity since $[h_C]$ is the inverse of a bonding map.
Thus $I$ is an isomorphism. 
\end{proof}

Theorem \ref{cech-cofinal} can also be deduced from Theorem \ref{main-poly} using that every resolution is an expansion
(see \cite{M00}*{Theorem \ref{book:exp-res}}) and every two expansions of the same space are pro-isomorphic
(see \cite{M00}*{Corollary \ref{book:expansion3}}).

On the other hand, as a byproduct of the proof of Theorem \ref{cech-cofinal} we obtain an alternative construction
of a sequential polyhedral resolution for a given coronated polyhedron.%
\footnote{Historically, this was the original construction (see version 1 of the arXiv preprint \cite{M-III}).}

\begin{corollary} \label{cech-cofinal'}
Every coronated polyhedron $X$ is the limit of a pre-compactohedral inverse sequence.
\end{corollary}

\begin{proof} This follows from the proof of Theorem \ref{cech-cofinal} by considering the nerves 
of the constructed covers.
To ensure that $\lim R_i=X$ and $\lim K_i=K$ one needs to modify the construction by setting each $C_n$ to be 
a finite subset of $B_{\lambda_n/2}$ (rather than $B_{\lambda_n}$) that covers $K$.
This guarantees that each $C_{i+1}$ barycentrically refines $C_i$.
The bonding maps $N_{E_{n+1}}\to N_{E_n}$ are defined in the usual way, by sending each vertex $U\in E_{n+1}$
to the barycenter of the simplex $\{V_1,\dots V_k\}$, where $V_1,\dots,V_k$ are those elements of $E_n$
that contain $U$, and extending linearly.
\end{proof}

\subsection{A simple construction in a special case}

For a finite-dimensional coronated polyhedron $X=K\cup P$, where $K$ is a compactum and the polyhedron $P=X\but K$ 
is locally compact, we also note a different representation of $X$ as the limit of a compactohedral 
inverse sequence.
Indeed, by the proof of Proposition \ref{cor-lc}(c) we may identify $X$ with a subset of $S^n$ for some $n$ 
such that $P$ is a subpolyhedron of an open subset $V$ of $S^n\but K$; now the assertion follows from

\begin{theorem} Suppose that $X\subset S^n$ is a coronated polyhedron, $X=K\cup P$, where the polyhedron $P$
is a subpolyhedron of an open subset $V$ of $S^n\but K$.
Then $X$ is the limit of a compactohedral inverse sequence of the form $\dots\subset R_1\subset R_0$, 
where each $R_i$ is a subpolyhedron of an open subset of $S^n$.
\end{theorem}

\begin{proof} Let $\dots K_1\subset K_0$ be a nested sequence of closed polyhedral neighborhoods of $K$ in $S^n$ 
such that each $K_i$ is a PL manifold with boundary, each $K_{i+1}\subset\Int K_i$ and $\bigcap_i K_i=K$.
We will show that $R_i\bydef P\cup (K_i\cap V)\cup\Int K_i$ is a subpolyhedron of the open subset $U_i\bydef V\cup\Int K_i$ 
of $S^n$ (hence, in particular, a polyhedron).
See Example \ref{resolution example} below for an illustration.

Let $M_i=\Cl(S^n\but K_i)$.
Since $(M_i,\partial M_i)$ is triangulated by a pair of subcomplexes of some triangulation of $S^n$, it is easy 
to see that $(M_i\cap V,\,\partial M_i\cap V)$ is triangulated by a pair of subcomplexes of some triangulation 
$T_V$ of $V$.
On the other hand, $P$ is also triangulated by a subcomplex of some triangulation $T'_V$ of $V$.
Then $P\cap M_i$ and $(M_i\cap V,\,\partial M_i\cap V)$ are triangulated respectively by a subcomplex $C_i$
and a pair of subcomplexes $(A_i,\partial A_i)$ of a common subdivision of $T_V$ and $T'_V$.

Next, $(U_i\cap K_i,\,U_i\cap\partial K_i)$ is triangulated by a pair of subcomplexes $(B_i,\partial B_i)$ of 
some triangulation of $U_i$.
Now $U_i\cap\partial K_i=V\cap\partial M_i$ is triangulated by $\partial A_i$ and by $\partial B_i$, and hence 
also by their common subdivision $\partial A_i'=\partial B_i'$, which extends to subdivisions $A'_i$ of $A_i$ 
and $B'_i$ of $B_i$.
Therefore $A'_i\cup B'_i$ is a triangulation of $U_i=(U_i\cap K_i)\cup (V\cap M_i)$ such that
its subcomplex $C'_i\cup B'_i$ triangulates $R_i=(U_i\cap K_i)\cup(P\cap M_i)$.

Since $K_{i+1}\cup P\subset R_i\subset K_i\cup P$, the inverse sequence $\dots\subset R_1\subset R_0$ is 
compactohedral.
Clearly $\bigcap_{i\in\N} R_i=X$.
\end{proof}

\begin{example} \label{resolution example}
Let $X=K\cup P\subset\R^2$, where $P=\{\frac1n\mid n\in\N\}\x[-1,1]$ and $K=\{0\}\x\{-1,1\}$.
Then $P$ is a subpolyhedron of the open set $V\bydef \R^2\but\{0\}\x[-1,1]$.
Let $K_n=[-\frac1n,\frac1n]\x\big([-1-\frac1n,\,-1+\frac1n]\cup[1-\frac1n,1+\frac1n]\big)$, 
the union of two closed $l_\infty$-balls of radius $\frac1n$ about the two points of $K$.
Then the subpolyhedra $R_n\bydef P\cup K_n\but\{0\}\x\{-1+\frac1n,1-\frac1n\}$ of
the open sets $U_n\bydef \R^2\but\{0\}\x[-1+\frac1n,1-\frac1n]$ form a compactohedral 
inverse sequence with $\bigcap_{i\in\N} R_i=X$.
\end{example}

\section{Homology} \label{homology}

\subsection{Homology of the extended telescope}

\begin{lemma}\label{two-five}
Let $X$ be the limit of a pre-compactohedral inverse sequence of polyhedra 
$\dots\xr{p_2} R_2\xr{p_1} R_1$, and let $H_*$ (resp.\ $H^*$) be a generalized (co)homology theory
satisfying map excision on closed pairs of coronated polyhedra.
Then the map of pairs given by the metric quotient map $R_{[0,\infty]}\to R_{[0,\infty)}^\Star$
induces isomorphisms 
\begin{gather*}
H_n(R_{[0,\infty]},\,X\cup R_0)\to H_n(R_{[0,\infty)}^\Star,\,\STAR\cup R_0),\\
H^n(R_{[0,\infty]},\,X\cup R_0)\to H^n(R_{[0,\infty)}^\Star,\,\STAR\cup R_0).
\end{gather*}
\end{lemma}

See Example \ref{telescope-quotient} concerning $R_{[0,\infty)}^\Star$ and $\STAR$.

\begin{proof} We prove only the isomorphism in homology, the case of cohomology being similar.
Let $K_i\subset L_i\subset R_i$ be the associated compact and closed subpolyhedra.

The map of triples
\[(R_{[0,\infty]},\,K_{[0,\infty]}\cup X\cup R_0,\,X\cup R_0)\to
(R_{[0,\infty)}^\Star,\,K_{[0,\infty)}^\Star\cup R_0,\,\STAR\cup R_0)\] 
given by metric quotient map $R_{[0,\infty]}\to R_{[0,\infty)}^\Star$ 
yields the following commutative diagram with exact columns:
\[\begin{CD}
\vdots@.\vdots\\
@VVV@VVV\\
H_i(K_{[0,\infty]},\,K\cup K_0)@>q_K>>H_i(K_{[0,\infty)}^\Star,\,\STAR\cup K_0)\\
@VVV@VVV\\
H_i(R_{[0,\infty]},\,X\cup R_0)@>q_R>>H_i(R_{[0,\infty)}^\Star,\,\STAR\cup R_0)\\
@VVV@VVV\\
H_i(R_{[0,\infty]},\,K_{[0,\infty]}\cup X\cup R_0)@>q_{(R,K)}>>
H_i(R_{[0,\infty)}^\Star,\,K_{[0,\infty)}^\Star\cup R_0)\\
@VVV@VVV\\
\vdots@.\vdots\\
\end{CD}\]
By the map excision axiom, $q_K$ is an isomorphism.
Then by the five-lemma, to show that $q_R$ is an isomorphism it suffices to show that $q_{(R,K)}$ is an isomorphism.

Let us recall that each $p_i(L_{i+1})\subset K_i\subset L_i$.
Let $L_{[0,\infty]}$ be the extended mapping telescope of $\dots\xr{p_1|_{L_1}}L_1\xr{p_0|_{L_0}}L_0$, identified 
with the corresponding subset of $R_{[0,\infty]}$; clearly, $L_\infty=L_{[0,\infty]}\cap X=K$.
The mapping cylinder of each $p_i|_{L_{i+1}}\:L_{i+1}\to K_i$ collapses onto the mapping cylinder of its 
restriction $p_i|_{K_{i+1}}\:K_{i+1}\to K_i$ (even though $L_{i+1}$ does not necessarily collapse onto $K_{i+1}$).
These collapses performed simultaneously yield a deformation retraction of $L_{[0,\infty]}$ onto $K_{[0,\infty]}$,
which also qualifies as a deformation retraction $r_t$ of $L_{[0,\infty]}\cup X\cup R_0$ onto 
$K_{[0,\infty]}\cup X\cup R_0$.
We may assume that $r_t$ is a strong deformation retraction (i.e.\ keeps $K_{[0,\infty]}\cup X\cup R_0$ fixed), and 
for each $x\in L_{[0,\infty]}$ and each $t\in I$ we have $d\big(r_t(x),X\big)\le 2d(x,X)$.
Since $R_{[0,\infty]}$ is an ANR (see \cite{M00}*{Theorem \ref{book:extended telescope}(b)}), $r_t$ extends to 
a homotopy $\bar r_t$ of $R_{[0,\infty]}$ in itself, which keeps $K_{[0,\infty]}\cup X\cup R_0$ fixed and satisfies
$d\big(r_t(x),X\big)\le 3d(x,X)$ for each $x\in R_{[0,\infty]}$ and each $t\in I$.
Since $\bar r_t$ keeps $X$ fixed, it descends to a homotopy of the metric quotient $R_{[0,\infty)}^\Star$ in itself;
indeed, the latter is continuous at $\STAR$ (in the topology of the metric quotient) due to the distance estimate
for $\bar r_t$. 
It follows that the pair $(R_{[0,\infty]},\,K_{[0,\infty]}\cup X\cup R_0)$ is homotopy equivalent to 
$(R_{[0,\infty]},\,L_{[0,\infty]}\cup X\cup R_0)$, and the pair 
$(R_{[0,\infty)}^\Star,\,K_{[0,\infty)}^\Star\cup R_0)$ is homotopy equivalent to
$(R_{[0,\infty)}^\Star,\,L_{[0,\infty)}^\Star\cup R_0)$.

Each pair $(R_i,L_i)$ is map-excision-equivalent to $\big(\Cl_{R_i}(R_i\but L_i),\,\Fr_{R_i}(R_i\but L_i)\big)$.
(Here map excision can be replaced by usual excision along with homotopy invariance.)
It follows from the definition of a pre-compactohedral inverse sequence that the latter pair 
is homeomorphic to 
$\Big(\Cl_X\big((p^\infty_i)^{-1}(R_i\but L_i)\big),\,\Fr_X\big((p^\infty_i)^{-1}(R_i\but L_i)\big)\Big)$, 
which is in turn map-excision-equivalent to $(X,\,M_i)$, where $M_i=\Cl_X\big((p^\infty_i)^{-1}(L_i)\big)$.
Finally, the latter pair is map-excision-equivalent to $(X/K,\,M_i/K)$, which will denote by $(X^\Bullet,M_i^\Bullet)$ 
for brevity.
Writing $M^\Bullet_{[0,\infty]}$ for the mapping telescope, we get a commutative diagram
\[\scalebox{0.8}{$\begin{CD}
H_i(R_{[0,\infty]},\,K_{[0,\infty]}\cup X\cup R_0)@>q_{(R,K)}>>
H_i\big(R_{[0,\infty)}^\Star,\,K_{[0,\infty)}^\Star\cup R_0\big)\\
@V\simeq VV@V\simeq VV\\
H_i(R_{[0,\infty]},\,L_{[0,\infty]}\cup X\cup R_0)@>q_{(R,Q)}>>
H_i\big(R_{[0,\infty)}^\Star,\,L_{[0,\infty)}^\Star\cup R_0\big)\\
@V\simeq VV@V\simeq VV\\
H_i(X\x[0,\infty],\,M_{[0,\infty]}\cup X\x\{0,\infty\})@>q_{(X,M)}>>
H_i\big(X\x[0,\infty]/X\x\{\infty\},\,M_{[0,\infty)}^\Star\x\{\infty\}\cup X\x\{0\}\big)\\
@V\simeq VV@V\simeq VV\\
H_i(X^\Bullet\x[0,\infty],\,M^\Bullet_{[0,\infty]}\cup X^\Bullet\x\{0,\infty\})@>q_{(X^\Bullet,M^\Bullet)}>>
H_i(X^\Bullet\x[0,\infty]/X^\Bullet\x\{\infty\},\,M^\Bullet_{[0,\infty]}\cup X^\Bullet\x\{0\}),
\end{CD}$}\]
where the vertical arrows in the top square are induced by the inclusions of pairs, which are homotopy equivalences;
those in the middle square are each composed of two map excision isomorphisms;
and those in the bottom square are map excision isomorphisms.
Thus showing that $q_{(R,K)}$ is an isomorphism amounts to showing that $q_{(X^\Bullet,M^\Bullet)}$ is an isomorphism.
Finally, the map of triples
\begin{multline*}
(X^\Bullet\x[0,\infty],\,M^\Bullet_{[0,\infty]}\cup X^\Bullet\x\{0,\infty\},\,X^\Bullet\x\{0,\infty\})\\ \to
(X^\Bullet\x[0,\infty]/X^\Bullet\x\{\infty\},\,M^\Bullet_{[0,\infty]}\cup X^\Bullet\x\{0\},\,\{X^\Bullet\x\{\infty\}\}\cup X^\Bullet\x\{0\})
\end{multline*} 
given by the metric quotient map $X^\Bullet\x[0,\infty]\to X^\Bullet\x[0,\infty]/X^\Bullet\x\{\infty\}$
yields the following commutative diagram with exact columns:
\[\scalebox{0.8}{$\begin{CD}
\vdots@.\vdots\\
@VVV@VVV\\
H_i(M^\Bullet_{[0,\infty]},\,\{K\}\x\{\infty\}\cup M_0)@>q_{M^\Bullet}>>
H_i(M^\Bullet_{[0,\infty]},\,\{X^\Bullet\x\{\infty\}\}\cup M_0)\\
@VVV@VVV\\
H_i(X^\Bullet\x[0,\infty],\,X^\Bullet\x\{0,\infty\})@>q_{X^\Bullet}>>
H_i(X^\Bullet\x[0,\infty]/X^\Bullet\x\{\infty\},\,\{X^\Bullet\x\{\infty\}\}\cup X^\Bullet\x\{0\})\\
@VVV@VVV\\
H_i(X^\Bullet\x[0,\infty],\,M^\Bullet_{[0,\infty]}\cup X^\Bullet\x\{0,\infty\})@>q_{(X^\Bullet,M^\Bullet)}>>
H_i(X^\Bullet\x[0,\infty]/X^\Bullet\x\{\infty\},\,M^\Bullet_{[0,\infty]}\cup X^\Bullet\x\{0\})\\
@VVV@VVV\\
\vdots@.\vdots
\end{CD}$}\]
Here $q_M$ is induced by a homeomorphism, and so is an isomorphism.
To see that $q_{X^\Bullet}$ is an isomorphism, we let us represent the underlying map of pairs in different 
notation, as a map
$(X^\Bullet\x[0,1],\,X^\Bullet\x\{0,1\})\to (X^\Bullet\x[0,2]/X^\Bullet\x\{2\},\,\{X^\Bullet\x\{2\}\}\cup X^\Bullet\x\{0\})$.
The latter factors into the composition
\begin{multline*}
(X^\Bullet\x[0,1],\,X^\Bullet\x\{0,1\})\xr{g}\big(X^\Bullet\x[0,2]/X^\Bullet\x\{2\},\,(X^\Bullet\x[1,2]/X^\Bullet\x\{2\})\cup X^\Bullet\x\{0\}\big)\\
\xr{h}(X^\Bullet\x[0,2]/X^\Bullet\x\{2\},\,\{X^\Bullet\x\{2\}\}\cup X^\Bullet\x\{0\}),
\end{multline*}
where $g$ is an inclusion of pairs 
and induces an isomorphism on homology by map excision (or alternatively by usual excision along with homotopy invariance), 
and $h$ is a homotopy equivalence of pairs
(since it is homotopic to the identity keeping $X^\Bullet\x\{0\}$ and the point $\{X^\Bullet\x\{2\}\}$ fixed), and 
hence induces an isomorphism on homology by the homotopy axiom.
Thus $q_{(X^\Bullet,M^\Bullet)}$ is an isomorphism by the five-lemma.
\end{proof}

\subsection{Strong shape invariance}

Theorem \ref{invariance} is a consequence of the following

\begin{theorem}\label{invariance2}
Let $H$ be a generalized (co)homology theory satisfying map excision 
on closed pairs of coronated polyhedra.
Then $H$ is an invariant of strong shape on coronated polyhedra, in the sense that its restriction 
to single spaces (i.e.\ pairs of the form $(X,\emptyset)$), regarded as a covariant (contravariant) functor 
from the homotopy category of coronated polyhedra to the category of graded abelian groups, 
factors through the strong shape category.
\end{theorem}

The compact case of Theorem \ref{invariance2} is due to Mrozik \cite{Mr2}.

\begin{proof}
We discuss only the case of homology, the case of cohomology being similar (and less interesting).
If $X$ and $Y$ are coronated polyhedra, by Theorem \ref{compactohedral}
$X$ admits a polyhedral resolution of the form $\dots\to R_2\to R_1$
and $Y$ admits a polyhedral resolution of the form $\dots\to Q_2\to Q_1$.
Then a strong shape morphism $F\:X\to Y$ can be represented by a map 
$f\:(R_{[0,\infty)},R_0)\to (Q_{[0,\infty)},Q_0)$ which extends to a continuous map 
$\bar f\:(R_{[0,\infty)}^\Star,\,\STAR\cup R_0)\to (Q_{[0,\infty)}^\Star,\,\STAR\cup Q_0)$
(see \cite{M00}*{Remark \ref{book:telescope0}}).
By Lemma \ref{two-five} $\bar f$ induces homomorphisms
$H_n(R_{[0,\infty]},\,X\cup R_0)\to H_n(Q_{[0,\infty]},\,Y\cup Q_0)$.
Since $R_{[0,\infty)}$ deformation retracts onto $R_0$, we have
$H_n(R_{[0,\infty]},R_0)=0$ for all $n$.
Hence the exact sequence of the triple yields
$H_n(R_{[0,\infty]},\,X\cup R_0)\simeq H_{n-1}(X\sqcup R_0,\,R_0)\simeq H_{n-1}(X)$,
and similarly
$H_n(Q_{[0,\infty]},\,Y\cup Q_0)\simeq H_{n-1}(Y)$.
Thus we get a homomorphism $\bar f_*\:H_{n-1}(X)\to H_{n-1}(Y)$.
Given another representative 
$\bar g\:(R_{[0,\infty)}^\Star,\,\STAR\cup R_0)\to (Q_{[0,\infty)}^\Star,\,\STAR\cup Q_0)$
of $F$, it is homotopic to $\bar f$, and hence $\bar f_*=\bar g_*$.
Thus $F_*\bydef \bar f_*$ is well-defined.
It is easy to see that $\id_*=1$ and $G_*F_*=(GF)_*$.
If $F$ is represented by a map $\phi\:X\to Y$ (in other words, $f$ extends 
to a continuous map $(R_{[0,\infty]},\,X\cup R_0)\to (Q_{[0,\infty]},\,Y\cup Q_0)$), 
then clearly $F_*=\phi_*$.
\end{proof}

\subsection{Milnor-type short exact sequence}

Theorem \ref{main-ses} is a consequence of the following

\begin{theorem}\label{main-ses2}
Let $\dots\to P_2\to P_1$ be an ANR resolution of a coronated polyhedron $X$ and let $H_*$ (resp.\ $H^*$) be 
a generalized (co)homology theory satisfying $\nullseq$-additivity on null-sequences of polyhedra and 
map excision on closed pairs of coronated polyhedra.

(a) There is an isomorphism $H^n(X)\simeq\colim H^n(P_i)$ and a short exact sequence
\[0\to\derlim H_{n+1}(P_i)\to H_n(X)\to\lim H_n(P_i)\to 0.\]

(b) These are natural with respect to any natural transformation of (co)homology theories and with respect 
to any map $f\:X\to X'$ along with a collection of maps $f_i\:P_{j_i}\to P'_i$ that fit in a homotopy 
commutative diagram with $f$ and with the bonding maps. 
\end{theorem}

\begin{proof}
First we prove all assertions for one specific sequential resolution of $X$.
Namely, by Theorem \ref{compactohedral} $X$ is the limit of a pre-compactohedral inverse sequence 
of polyhedra $\dots\xr{p_2} R_2\xr{p_1} R_1$, which by Theorem \ref{resolution theorem} is a resolution of $X$.

By standard arguments (see \cite{M00}*{Proposition \ref{book:main-ses0}}) there is a short exact sequence
\[0\to\derlim H_{n+1}(R_i)\to H_{n+1}(R_{[0,\infty)}^\Star,\,\STAR\cup R_0)\to\lim H_n(R_i)\to 0\] 
and an isomorphism $H^{n+1}(R_{[0,\infty)}^\Star,\,\STAR\cup R_0)\simeq\colim H^n(P_i)$
which are natural with respect to any natural transformation of (co)homology theories and with respect to any 
collection of maps $f_i\:R_{j_i}\to R'_i$ that fit in a homotopy commutative diagram with the bonding maps. 

By Lemma \ref{two-five} the metric quotient map $R_{[0,\infty]}\to R_{[0,\infty)}^\Star$ induces an isomorphism
\[H_{n+1}(R_{[0,\infty)}^\Star,\,\STAR\cup R_0)\simeq H_{n+1}(R_{[0,\infty]},\,X\cup R_0).\]
Since $R_{[0,\infty]}$ deformation retracts onto $R_0$, from the exact sequence of the pair
$(R_{[0,\infty]},\,R_0)$ we have $H_i(R_{[0,\infty]},\,R_0)=0$ for all $i$.
Hence the connecting homomorphism 
\[H_{n+1}(R_{[0,\infty]},\,X\cup R_0)\xr{\partial_*} H_n(X\cup R_0,\,R_0)\] from the exact sequence 
of the triple $(R_{[0,\infty]},\,X\cup R_0,\,R_0)$ is an isomorphism.
By excision $H_n(X\cup R_0,\,R_0)\simeq H_n(X)$.
Thus we obtain a short exact sequence
\[0\to\derlim H_{n+1}(R_i)\to H_n(X)\to\lim H_n(R_i)\to 0\]
and by similar arguments also an isomorphism $H^n(X)\simeq\colim H^n(R_i)$.
By inspecting their constructions one can see that they is natural with respect to any natural transformation
of homology theories.

Let $f\:X\to X'$ be a map into another coronated polyhedron (not necessarily respecting the decompositions
$X=K\cup P$ and $X'=K'\cup P'$) and let $\dots\xr{p'_2} R'_2\xr{p'_1} R'_1$ be a resolution of $X'$ which is 
a pre-compactohedral inverse sequence.
Since $\dots\xr{p_2} R_2\xr{p_1} R_1$ is a resolution of $X$ and each $R'_i$ is an ANR, the map $f$ extends to 
a continuous map of the extended mapping telescopes $F\:R_{[0,\infty]}\to R'_{[0,\infty]}$ such that 
$F^{-1}(R'_i)=R_{j_i}$ for some increasing sequence $j_1,j_2,\dots$, and in particular $F^{-1}(X')=X$
(see \cite{M00}*{Theorem \ref{book:strongshape}}).
Then we have a commutative diagram
\[\begin{CD}
R_{[0,\infty]}@>F>>R'_{[0,\infty]}\\
@VVV@VVV\\
R_{[0,\infty)}^\Star@>F/X>>(R')_{[0,\infty)}^\Star\\
\end{CD}\]
and hence the isomorphism of Lemma \ref{two-five} commutes with $F_*$ and $(F/X)_*$.
From this we straightforwardly get commutative diagrams
\[\begin{CD}
0@>>>\derlim H_{n+1}(R_i)@>>>H_n(X)@>>>\lim H_n(R_i)@>>>0\\
@.@VVV@VVV@VVV@.\\
0@>>>\derlim H_{n+1}(R'_i)@>>>H_n(X')@>>>\lim H_n(R'_i)@>>>0\\
\end{CD}\]
and
\[\begin{CD}
H^n(X)@>\simeq>>\colim H^n(R_i)\\
@VVV@VVV\\
H^n(X')@>\simeq>>\colim H^n(R'_i).\\
\end{CD}\]

Given an arbitrary sequential ANR resolution $\dots\to P_2\to P_1$ of $X$, we similarly get continuous maps
$P_{[0,\infty]}\to R_{[0,\infty]}$ and $R_{[0,\infty]}\to P_{[0,\infty]}$ both extending $\id_X$
whose compositions are the identities up to the bonding maps.
These yield isomorphisms $\xi_*\:\lim H_n(P_i)\simeq\lim H_n(R_i)$, 
$\xi_*^1\:\derlim H_n(P_i)\simeq\derlim H_n(R_i)$ and $\xi^*\:\colim H^n(P_i)\simeq\colim H^n(R_i)$.

Given a map $f\:X\to X'$ and an arbitrary sequential resolution $\dots\to P_2'\to P_1'$ of $X'$, we also get 
a continuous extension $P_{[0,\infty]}\to P'_{[0,\infty]}$ of $f$, which induces homomorphisms
$f_*\:\lim H_n(P_i)\to\lim H_n(P'_i)$, $f_*^1\:\derlim H_n(P_i)\to\derlim H_n(P'_i)$ and 
$f^*\:\colim H^n(P'_i)\to\colim H^n(P_i)$.
Moreover, the diagram
\[\begin{CD}
P_{[0,\infty]}@>>>P'_{[0,\infty]}\\
@VVV@VVV\\
R_{[0,\infty]}@>>>R'_{[0,\infty]}\\
\end{CD}\]
homotopy commutes up to the bonding maps.
Hence these homomorphisms $f_*$, $f_*^1$ and $f^*$ commute with the isomorphisms $\xi_*$, $\xi_*^1$ and $\xi^*$.

From this we get an isomorphism $H^n(X)\simeq\colim H^n(P_i)$ and a short exact sequence
\[0\to\derlim H_{n+1}(P_i)\to H_n(X)\to\lim H_n(P_i)\to 0\]
which are natural with respect to to any map $f\:X\to X'$ along with a collection of maps $f_i\:P_{j_i}\to P'_i$ 
that fit in a homotopy commutative diagram with $f$ and with the bonding maps. 
By inspecting their construction one can see that they are also natural with respect to any natural 
transformation of (co)homology theories.
\end{proof}

\subsection{Uniqueness theorem}

\begin{corollary} \label{uniqueness}
Let $h_*$ (resp.\ $h^*$) be a generalized (co)homology theory on closed pairs of coronated polyhedra which 
satisfies the $\nullseq$-additivity axiom for null-sequences of polyhedra and the map excision axiom.
Then any natural transformation of $h_*$ (resp.\ $h^*$) into another such theory is an isomorphism as long as 
it is an isomorphism on polyhedra.
\end{corollary}

Theorem \ref{uniqueness2} follows from Corollary \ref{uniqueness} and Milnor's theorem that
if $h_*$ (resp.\ $h^*$) satisfies $\bigsqcup$-additivity for disjoint unions of compact polyhedra,
then any natural transformation of $h_*$ (resp.\ $h^*$) into another such theory is an isomorphism on polyhedra
as long as it is an isomorphism on $pt$ (see \cite{M00}*{Theorem \ref{book:uniqueness1}(b)}).

\begin{proof}
Let $X$ be a coronated polyhedron.
By Theorem \ref{compactohedral} $X$ is the limit of a pre-compactohedral inverse sequence 
of polyhedra $\dots\xr{p_2} R_2\xr{p_1} R_1$, which by Theorem \ref{resolution theorem} is a resolution of $X$.

In the case of homology, by Theorem \ref{main-ses2} $t$ yields a commutative diagram
$$\begin{CD}
0@>>>\derlim h_{n+1}(R_i)@>>>h_n(X)@>>>\lim h_n(R_i)@>>>0\\
@.@V\simeq VV@VtVV@V\simeq VV@.\\
0@>>>\derlim k_{n+1}(R_i)@>>>k_n(X)@>>>\lim k_n(R_i)@>>>0
\end{CD}$$
where the outer vertical arrows are isomorphisms. 
Hence by the five-lemma the central vertical arrow is also an isomorphism. 

In the case of cohomology, by Theorem \ref{main-ses2} $t$ yields a commutative diagram
$$\begin{CD}
h^n(X)@>\simeq>>\colim h^n(R_i)\\
@VtVV@V\simeq VV\\
k^n(X)@>\simeq>>\colim k^n(R_i).
\end{CD}$$
Hence $t\:h^n(X)\to k^n(X)$ is an isomorphism.
\end{proof}

\subsection*{Disclaimer}

I oppose all wars, including those wars that are initiated by governments at the time when 
they directly or indirectly support my research. The latter type of wars include all wars 
waged by the Russian state in the last 25 years (in Chechnya, Georgia, Syria and Ukraine) 
as well as the USA-led invasions of Afghanistan and Iraq.

\end{document}